\documentclass[a4paper,11pt,leqno]{article}

\usepackage[latin1]{inputenc}
\usepackage[T1]{fontenc} 
\usepackage[english]{babel}
\usepackage{verbatim}
\usepackage{amsfonts}
\usepackage{amsmath}
\usepackage{amsthm}
\usepackage{amssymb}
\usepackage{dsfont}
\usepackage{color}
\usepackage{graphicx}
\usepackage{bm}

\theoremstyle{definition}
\newtheorem{defin}{Definition}[section]

\newtheorem{rem}[defin]{Remark}

\theoremstyle{plane}
\newtheorem{thm}[defin]{Theorem}
\newtheorem{prop}[defin]{Proposition}
\newtheorem{coroll}[defin]{Corollary}
\newtheorem{lemma}[defin]{Lemma}

\newcommand{\mbb}{\mathbb}

\newcommand{\mc}{\mathcal}
\newcommand{\mf}{\mathfrak}
\newcommand{\veps}{\varepsilon}
\newcommand{\what}{\widehat}
\newcommand{\wtilde}{\widetilde}
\newcommand{\vphi}{\varphi}
\newcommand{\oline}{\overline}

\newcommand{\ra}{\rightarrow}

\newcommand{\hra}{\hookrightarrow}

\newcommand{\g}{\gamma}
\newcommand{\lan}{\langle}
\newcommand{\ran}{\rangle}

\newcommand{\R}{\mathbb{R}}

\newcommand{\N}{\mathbb{N}}
\newcommand{\Z}{\mathbb{Z}}

\renewcommand{\div}{{\rm div}\,}

\newcommand{\Id}{{\rm Id}\,}

\allowdisplaybreaks

\def\d{\partial}
\def\div{{\rm div}\,}

\title{\Large{\textbf{\textsc{A singular limit problem for rotating capillary fluids with variable rotation axis}}}}

\author{\textsl{Francesco Fanelli}$\,$\footnote{Present address: \textit{Institut Camille Jordan, UMR CNRS 5208}, \textsc{Universit\'e Claude Bernard - Lyon 1};
B\^atiment Braconnier; 48, Boulevard du 11 novembre 1918 -- 69622 Villeurbanne cedex - FRANCE; email: \texttt{fanelli@math.univ-lyon1.fr}.} \vspace{.2cm} \\
{\small \textit{Centro di Ricerca Matematica ``Ennio De Giorgi''}} \vspace{.1cm} \\
{\small \textsc{Scuola Normale Superiore}} \vspace{.2cm} \\
{\small {Collegio Puteano}} \\
{\small {Piazza dei Cavalieri, 3}} \\
{\small {I-56126 -- Pisa, ITALY}} \vspace{.3cm} \\
{\small \ttfamily{francesco.fanelli@sns.it}} }

\vspace{.3cm}
\date\today

\begin{document}

\maketitle

\subsubsection*{Abstract}
{\small In the present paper we study a singular perturbation problem for a Navier-Stokes-Korteweg model with Coriolis force.
Namely, we perform the incompressible and fast rotation asymptotics simultaneously, while we keep the capillarity coefficient constant in
order to capture surface tension effects in the limit.

We consider here the case of variable rotation axis: we prove the convergence to a linear parabolic-type equation with variable coefficients.
The proof of the result relies on compensated compactness arguments.

Besides, we look for minimal regularity assumptions on the variations of the axis.}

\paragraph*{2010 Mathematics Subject Classification:}{\small 35Q35 
(primary); 35B25, 
35B40, 
76U05 
(secondary).}

\paragraph*{Keywords:}{\small Navier-Stokes-Korteweg system; singular perturbation problem; low Mach, Rossby and Weber numbers;
variable rotation axis; admissible modulus of continuity.}

\section{Introduction} \label{s:intro}

In the present paper we continue the study, started in \cite{F}, of singular perturbation problems for viscous capillary
fluids under the action of fast rotation of the Earth.

Denoting by $\rho\geq0$ the density of the fluid and by $u\in\R^3$ its velocity field, the mathematical model is given by the
Navier-Stokes-Korteweg system with Coriolis force
$$
\begin{cases}
{\rm St}\,\d_t\rho+\div\left(\rho u\right)\,=\,0 \\[1ex]
{\rm St}\,\d_t\!\left(\rho u\right)+\div\bigl(\rho u\otimes u\bigr)+
\dfrac{1}{{\rm Fr}^2}\nabla\Pi(\rho)-\dfrac{\nu}{{\rm Re}}\div\bigl(\rho Du\bigr)-
\dfrac{\kappa}{{\rm We}}\rho\nabla\Delta\rho+\dfrac{\mf{C}(\rho,u)}{{\rm Ro}}\,=\,0\,,
\end{cases} \leqno(N\!S\!K)
$$
where the positive real numbers $\nu$ and $\kappa$ are respectively the viscosity and capillarity coefficients,
$Du$ is the viscous stress tensor and $\rho\nabla\Delta\rho$ is the surface tension tensor; finally,
$\Pi$ represents the pressure of the fluid, which is supposed to be a smooth function of the  density only.

In system $(N\!S\!K)$, the positive parameters ${\rm St}$, ${\rm Fr}$, ${\rm Re}$, ${\rm We}$ and ${\rm Ro}$ denote respectively the
Strouhal, Froude, Reynolds, Weber and Rossby numbers (see e.g. \cite{B-D-Met}, \cite{F-N}); we identify the Froude
and the Mach numbers.
The term $\mf{C}(\rho,u)=\mf{c}\,e^3\times\rho u$ represents the Coriolis operator,
where $\mf{c}$ is a suitably smooth function (we refer to Section \ref{s:results} below for the precise assumptions).

There are two important features to point out in the mathematical theory of equations $(N\!S\!K)$. First of all, the capillarity term 
in the momentum equation gives additional bounds for higher order derivatives of the density. Such a property shows up not just in
the classical energy inequality, but also through the so called \emph{BD entropy conservation}, a second energy inequality first discovered
in \cite{B-D_2004} by Bresch and Desjardins (see also \cite{B-D-L}) for our system, and then generalized by the same authors to different
models: see e.g. \cite{B-D_2003}, \cite{B-D_2007}. It turns out that the BD entropy structure is a fundamental ingredient in the theory
of compressible fluids with density-dependent viscosity coefficients: for instance, we quote here works \cite{B-G-Zat},
\cite{Kit-Laur-Niet}, \cite{J}, \cite{Mel-V}, \cite{Mu-Po-Za}.

On the other hand, one has to remark that the viscosity term is degenerate in regions of vacuum, where one loses then any information
on the velocity field and its gradient. For this reason, for system $(N\!S\!K)$ supplemented with a classical barotropic
pressure law (this hypothesis was however a bit relaxed), in \cite{B-D-L} Bresch, Desjardins and Lin proved the global in time
existence of a modified notion of weak solutions. Namely, under these assumptions stability can be obtained only in non-void regions
of the physical space, so that one has to test the equations only against functions supported in zones where the density does not vanish.
This was achieved by (formally) using $\rho\psi$ as a new test function, with $\psi\in\mc{D}_{t,x}$ and $\rho$ the density
itself, and passing to the limit in the new ``weak'' formulation.

However, it turns out that stability can be recovered in presence of additional terms in the momentum equation:
for instance, some drag forces (like in \cite{B-D_2003} by Bresch and Desjardins for a $2$-D shallow water model),
or a ``cold component'' in the pressure law (like in \cite{B-D_2007} by the same authors, where variations of the temperature are taken
into account as well); alternatively, it is possible to impose some additional integrability conditions on the initial velocity, as done in
\cite{Mel-V} by Mellet and Vasseur.
We refer to paper \cite{B-D-Zat} 
for a complete and interesting discussion on this subjet, as well as for some recent develpments (see also \cite{V-Yu}).

Differently to what done in \cite{F}, where the weak formulation of \cite{B-D-L} was used, we suppose here that the pressure term $\Pi$
is given by the sum of two components, a classical one $P(\rho)=\rho^\g$ and a singular one $P_c(\rho)=-\rho^{-\g_c}$,
for two suitable parameters $\g>0$, $\g_c>0$. Therefore (thanks to the results of \cite{B-D-Zat}, under much more general assumptions
than the ones considered here) one recovers existence of global in time weak solutions in the classical sense; let us point out that,
alternatively, we might have added some drag terms to our equations (as done in \cite{B-D_2003}), without substantially changing the
subsequent analysis.
The term $P_c$ is often referred to as \emph{cold pressure} (see e.g. \cite{B-D_2007}), because it is associated with the zero Kelvin
isothermal curve for heat conducting fluids; also, singular pressures naturally appears when e.g. Van der Waals type laws are considered.
In fact, for densities and temperatures close to $0$ the properties of the medium drastically change, damaging the validity
of the equations of motion: the presence of $P_c$ may be seen as a way of preserving stability of the model. At the mathematical
level, this term gives a control for negative powers of the density, which can be used to deduce  integrability properties on the velocity
field. We refer to Subsection 3.1 of \cite{B-D_2007} for more details and some physical insights. Let us also recall other works involving
singular pressure laws: for instance, \cite{Mu-Po-Za} for mixture of fluids with chemical reactions, \cite{Kit-Laur-Niet} for some
lubrication models in one space dimension and the above mentioned work \cite{B-D-Zat} for compressible Navier-Stokes equations.

In the sequel, we also assume ${\rm St}={\rm Re}=1$ in system $(N\!S\!K)$, and we set $\kappa=1$ for convenience.
Moreover, for $\veps\in\,]0,1]$ we take ${\rm Fr}={\rm Ro}=\veps$ and ${\rm We}=\veps^{2(1-\alpha)}$, where $\alpha\in[0,1]$:
we end up with the system
\begin{equation} \label{i_eq:NSK}
\begin{cases}
\d_t\rho+\div\left(\rho u\right)\,=\,0 \\[1ex]
\d_t\left(\rho u\right)+\div\bigl(\rho u\otimes u\bigr)+\dfrac{1}{\veps^2}\,\nabla\Pi(\rho)+
\dfrac{1}{\veps}\,\mf{C}(\rho,u)-\nu\div\bigl(\rho Du\bigr)-
\dfrac{1}{\veps^{2(1-\alpha)}}\,\rho \nabla\Delta\rho\,=\,0\,.
\end{cases}
\end{equation}
We are interested in studying the asymptotic behavior of a family of weak solutions $\bigl(\rho_\veps,u_\veps\bigr)_\veps$ for $\veps\ra0$.
This means that we are performing the incompressible and fast rotation limits simultaneously, combining them with
effects coming either from vanishing capillarity (for $0<\alpha\leq1$) or constant capillarity (i.e. $\alpha=0$).

The mathematical study of fluids in fast rotation has now a quite long history, which goes back to the pioneering paper \cite{B-M-N}
by Babin, Mahalov and Nicolaenko for the incompressible Navier-Stokes equations. We refer to book \cite{C-D-G-G} and the references
therein for a complete analysis of the problem for incompressible viscous fluids; in the matter of this, we also quote here
papaer \cite{B-D-GV_2004} and \cite{G-SR_2006}.
Let us just point out an important aspect in the theory of rotating fluids (see e.g. \cite{C-D-G-G}): the strong Coriolis force has a
stabilazing effect on the motion. Namely, in the limit $\veps\ra0$, the dynamics becomes constant in the direction parallel to the
rotation axis: the fluid moves along vertical coloumns (the so called ``Taylor-Proudman coloumns'') and the flow is purely $2$-dimensional,
evolving on a plane orthogonal to the rotation axis. We refer also to book \cite{Ped} for some physical background about this problem.

In the compressible case, various models have been considered, for which the choice of the scaling ${\rm Fr}={\rm Ro}=\veps$ is usually
assumed. The classical barotropic Navier-Stokes equations with constant viscosity coefficient (not depending on the density, then)
were studied in \cite{F-G-N} by Feireisl, Gallagher and Novotn\'y, and in \cite{F-G-GV-N} by the same authors in collaboration
with G\'erard-Varet. For general ill-prepared initial data, in the former article it was proved the convergence of the model to a $2$-D
quasi-geostrophic equation, by resorting to the spectral analysis of the (constant coefficient) singular perturbation operator.
In the latter work, instead, still for ill-prepared data, the limit system was identified as a linear equation, due to
the presence of the centrifugal force having the same scaling as the other quantities.
Also, in \cite{F-G-GV-N} the singular perturbation operator is no more constant coefficients, and the
authors had to resort to compensated compactness arguments in order to pass to the limit in the non-linear terms. This method
was firstly introduced by P.-L. Lions and Masmoudi (see \cite{L-M}-\cite{L-M_CRAS}) in dealing with incompressible limit problems,
and borrowed in \cite{G-SR_2006} by Gallagher and Saint-Raymond in the context of rotating fluids.

Concerning systems having a similar structure as \eqref{i_eq:NSK}, in the above mentioned paper \cite{B-D_2003} Bresch and Desjardins
studied the problem for a $2$-D viscous shallow-water model with density dependent viscosities (we refer also to \cite{G-SR_Mem} and
\cite{G_2008} for a similar system, but with constant viscosity coefficient) and with additional laminar and turbolent friction terms;
also, they assumed a vanishing capillarity regime. In \cite{J-L-W}, instead, J\"ungel, Lin and Wu considered more general viscosity
and capillarity tensors, working in a strong solutions framework, still in space dimension $2$ and in the vanishing capillarity
regime. In both works \cite{B-D_2003} and \cite{J-L-W}, for well-prepared initial data, the authors recovered convergence to a
quasi-geostrophic equation, by use of the modulated energy method.

Coming back to the somehow simpler form \eqref{i_eq:NSK}, in \cite{F} we studied the problem in the $3$-D domain $\Omega=\R^2\times\,]0,1[\,$
and for general ill-prepared initial data: the improvement was due to the use of spectral analysis tools (namely RAGE Theorem
and microlocal symmetrization arguments), as done in \cite{F-G-N}. We payed attention both to the vanishing and constant capillarity cases:
in the former instance we recovered the asymptotic result of \cite{B-D_2003} and \cite{J-L-W}, while in the latter we found the
convergence to a slightly modified $2$-D quasi-geostrophic equation: since surface tension effects do not
vanish in the limit, then they come into play also in the final relation.

\medbreak
In the present paper we want to continue the previous study, mainly focusing on the case of effectively variable rotation axis
depending just on the ``horizontal variables'',
in analogy to what done in \cite{G-SR_2006} by Gallagher and Saint-Raymond for the homogeneous incompressible Navier-Stokes system.
Indeed, when $\mf{c}\equiv1$ (so that $\mf{C}(\rho,u)=e^3\times\rho u$) the analysis of \cite{F} still applies: the presence of
the cold pressure is important just for stability of weak solutions, but it does not affect the singular pertubation problem
(actually, it simplifies things, since it supplies informations for $u$ and its gradient).
We will come back in a while to the hypotheses for the function $\mf{c}$. For the moment, let us point out that we restrict our
attention to the constant capillarity regime, corresponding to $\alpha=0$: we are interested here in capturing surface tension effects
in the limit; on the other hand, different choices of $\alpha$ can be treated in a similar way.

For non-constant rotation axis, the singular perturbation operator becomes variable coefficients, so that spectral analysis tools
are no more available. Resorting to the techniques of \cite{G-SR_2006}, also used in \cite{F-G-GV-N} for non-constant density profiles,
the idea is hence to apply compensated compactness arguments to prove the convergence in the non-linear terms: more precisely,
after a regularization procedure and integration by parts, we take advantage of the structure of our system to find
special algebraic cancellations and relations, which enable us to pass to the limit. The main novelty here is the presence of an additional
non-linear term, due to capillarity: the BD entropy structure gives compactness in space for
the density, but we miss uniform informations in time, so that we cannot pass
directly to the limit in it. Nonetheless, it turns out that this item exactly cancels out with another one, coming from
the analysis of the convective term. Notice also that the regularization process itself presents some complications with respect
to \cite{G-SR_2006} and \cite{F-G-GV-N}, because we have here less available controls for the velocity field and its gradient.
In the end, we can prove the convergence to a variable coefficients linear equation (of parabolic type) for the limit density,
which can be seen as a sort of stream function for the limit velocity field.

As it was already the case in \cite{G-SR_2006}, the previous arguments work under high regularity assumptions on the variation $\mf{c}$
of the axis. First of all, a control on the gradient of $\mf{c}$ is needed in the computations, in order to use the vorticity equation and
to decompose the horizontal part of the (approximated) velocity field into the basis given by $\nabla\mf{c}$ and its orthogonal
$\nabla^\perp\mf{c}$ (recall that $\mf{c}$ just depends on the horizontal variables). Hence, having $\mf{c}\in W^{1,\infty}$ seems to be
a necessary hypothesis, at least for this strategy to work; besides, we will also need to assume that $\nabla\mf{c}$ has non-degenerate
critical points, in a precise sense (as already done in \cite{G-SR_2006}).

But this is not all: the regularization procedure creates
some remainder terms (essentially, commutators between a smoothing operator and the variable coefficient), which we need to be small
together with their gradient in order to be able to exploit the vorticity formulation. This requirement asks for additional regularity
on $\nabla\mf{c}$: in the present paper we look for minimal smoothness assumptions on it in order to recover the convergence.
More precisely, we show that it is sufficient for $\nabla\mf{c}$ to be $\mu$-continuous, for some admissible modulus of continuity $\mu$.
The proof of this result relies just on fine commutator estimates, which can be obtained, nonetheless, in a quite classical way.

\medbreak
To conclude, we give an overview of the paper.
In the next section we set up the problem, formulating our working hypotheses, and we state the main results.
Section \ref{s:bounds} is devoted to establish suitable a priori estimates, in the general case of system $(N\!S\!K)$; besides, this
analysis allows to justify a technical requirement in \cite{F} for the vanishing capillarity regime.
In Section \ref{s:singular} we prove our main result, about the singular limit problem in the case of variable rotation axis, when
assumptions are made on the first variation of $\nabla\mf{c}$.
Finally, in the Appendix we recall some basic notions from Littlewood-Paley theory (Appendix \ref{app:LP}) and about
admissible moduli of continuity (Appendix \ref{app:continuity}); we postpone the proof of some technical results for the BD entropy
in Appendix \ref{app:BD}.


\paragraph{Notations.}
Let us introduce some notations here.

We will decompose $x\in\Omega\,:=\,\R^2\times\,]0,1[\,$ into $x=(x^h,x^3)$, with $x^h\in\R^2$ denoting its horizontal component. Analogously,
for a vector-field $v=(v^1,v^2,v^3)\in\R^3$ we set $v^h=(v^1,v^2)$, and we define the differential operators
$\nabla_h$ and $\div_{\!h}$ as the usual operators, but acting just with respect to $x^h$.
Finally, we define the operator $\nabla^\perp_h\,:=\,\bigl(-\d_2\,,\,\d_1\bigr)$.

Moreover, since we will reconduct ourselves to a periodic problem in the $x^3$-variable (see Remark \ref{r:period-bc} below),
we also introduce the following  decomposition: for a vector-field $X$, we write
\begin{equation} \label{dec:vert-av}
X(x)\,=\,\langle X\rangle(x^h)\,+\,\wtilde{X}(x)\,,\qquad\qquad\mbox{ where }\qquad
\langle X\rangle(x^h)\,:=\,\int_{\mbb{T}}X(x^h,x^3)\,dx^3\,.
\end{equation}
Notice that $\wtilde{X}$ has zero vertical average, and therefore we can write $\wtilde{X}(x)\,=\,\d_3\wtilde{Z}(x)$,
with $\wtilde{Z}$ having zero vertical average as well.
We also set $\wtilde{Z}\,=\,\mc{I}(\wtilde{X})\,=\,\d_3^{-1}\wtilde{X}$.

\subsubsection*{Acknowledgements}

The author wishes to express his gratitude to D. Bresch, for very interesting discussions and explanations about cold pressure laws
and the theory of compressible fluids with density-dependent viscosities, and for suggesting him to consider the problem of
regularity of the rotation axis. He is also deeply grateful to I. Gallagher for fruitful discussions on fluids in fast rotation
and valuable remarks on this work.

The author sincerely thanks the anonymous referee for his/her corrections and comments, which helped to improve the quality and clarity of the paper.

The author is member of the Gruppo Nazionale per l'Analisi Matematica, la Probabilit\`a
e le loro Applicazioni (GNAMPA) of the Istituto Nazionale di Alta Matematica (INdAM).

\section{Basic definitions and results} \label{s:results}
We describe here our main hypotheses and results.

\subsection{General setting}
Let us consider the rescaled Navier-Stokes-Korteweg system
\begin{equation} \label{eq:NSK}
\begin{cases}
\d_t\rho\,+\,\div\left(\rho u\right)\,=\,0 \\[1ex]
\d_t\left(\rho u\right)\,+\,\div\bigl(\rho u\otimes u\bigr)\,+\,
\dfrac{1}{{\rm Fr}^2}\,\nabla\Pi(\rho)\,-\,\nu\,\div\bigl(\rho Du\bigr)\,-\,
\dfrac{1}{{\rm We}}\,\rho\nabla\Delta\rho\,+\,\dfrac{1}{{\rm Ro}}\,\mf{C}(\rho,u)\,=\,0\,
\end{cases}
\end{equation}
in $\R_+\times\Omega$, where $\Omega$ is the infinite slab
$$
\Omega\,:=\,\R^2\,\times\,\,]0,1[\,.
$$
In the previous system, the scalar function $\rho\geq0$ represents the density of the fluid, while $u\in\R^3$ its velocity field.
The quantity $\nu$, which we will assume always to be a fixed positive constant, is the viscosity coefficient and
the operator $D$ denotes the viscous stress tensor, defined by
$$
Du\,:=\,\frac{1}{2}\,\bigl(\nabla u\,+\,^t\nabla u\bigr)\,.
$$
Moreover, we denoted by $\mf{C}$ the Coriolis operator, which takes into account the Earth rotation: here, we suppose that
$\mf C$ is given by
\begin{equation} \label{eq:coriolis}
\mf{C}(\rho,u)\,:=\,\mf{c}(x^h)\,e^3\,\times\,\rho\,u\,,
\end{equation}
where $e^3=(0,0,1)$ is the unit vector directed along the $x^3$-coordinate and $\mf{c}$ is a smooth scalar
function of the horizontal variables only. Finally, the function $\Pi(\rho)$ represents the pressure
of the fluid: for the reasons we explained in the introduction, we suppose here $\Pi\,=\,P\,+\,P_c$, where $P$ is the classical
pressure, given by the Boyle law
\begin{equation} \label{eq:def_P}
P(\rho)\,:=\,\frac{1}{2\g}\,\rho^\g\,,\qquad\qquad\mbox{ for some }\quad1\,<\,\g\,\leq\,2\,,
\end{equation}
and the second term is the \emph{cold pressure} component, for which we take the power law
\begin{equation} \label{eq:def_cold}
P_c(\rho)\,:=\,-\,\frac{1}{2\g_c}\,\rho^{-\g_c}\,,\qquad\qquad\mbox{ with }\qquad 1\,\leq\,\g_c\,\leq\,2\,.
\end{equation}
The presence of the $1/2$ is just a normalization in order to have $\Pi'(1)=1$: this fact simplifies some computations in the sequel.

In view of the analysis of the singular perturbation problem, we supplement system \eqref{eq:NSK} by complete slip boundary conditions,
in order to avoid the appearing of boundary layers effects.
Namely, if we denote by $n$ the unitary outward normal to the boundary $\d\Omega$ of the domain (simply,
$\d\Omega=\{x^3=0\}\cup\{x^3=1\}$), we impose
\begin{equation} \label{eq:bc}
\left(u\cdot n\right)_{|\d\Omega}\,=\,u^3_{|\d\Omega}\,=\,0\,,\qquad
\left(\nabla\rho\cdot n\right)_{|\d\Omega}\,=\,\d_3\rho_{|\d\Omega}\,=\,0\,,\qquad
\bigl((Du)n\times n\bigr)_{|\d\Omega}\,=\,0\,.
\end{equation}

\begin{rem} \label{r:period-bc}
Equations \eqref{eq:NSK}, supplemented by complete slip boundary boundary conditions,
can be recasted as a periodic problem with respect to the vertical variable, in the new domain
$$
\Omega\,=\,\R^2\,\times\,\mbb{T}^1\,,\qquad\qquad\mbox{ with }\qquad\mbb{T}^1\,:=\,[-1,1]/\sim\,,
$$
where $\sim$ denotes the equivalence relation which identifies $-1$ and $1$. Indeed, the equations are invariant if we extend
$\rho$ and $u^h$ as even functions with respect to $x^3$, and $u^3$ as an odd function.

In what follows, we will always assume that such modifications have been performed on the initial data, and
that the respective solutions keep the same symmetry properties.
\end{rem}

\subsection{Existence of weak solutions}

Here we will always consider initial data $(\rho_0,u_0)$ such that $\rho_0\geq0$ and
\begin{equation} \label{eq:initial}
\begin{cases}
\dfrac{1}{\rm Fr}\left(\rho_0\,-\,1\right)\;\in\;L^\g(\Omega)\qquad\mbox{ and }\qquad
\dfrac{1}{\rm Fr}\left(\dfrac{1}{\rho_0}\,-\,1\right)\;\in\;L^{\g_c}(\Omega) \\[1ex]
\sqrt{\rho_0}\,u_0\;,\;\nabla\sqrt{\rho_0}\;,\;\dfrac{1}{\sqrt{\rm We}}\nabla\rho_0\quad\in\;L^2(\Omega)\,.
\end{cases}
\end{equation}
At this point, let us also introduce the internal energy functions $h(\rho)$ and $h_c(\rho)$, such that
$$
\begin{cases}
h''(\rho)\,=\,\dfrac{P'(\rho)}{\rho}\,=\,\rho^{\g-2}\qquad\qquad\mbox{ and }\qquad\qquad
h(1)\,=\,h'(1)\,=\,0\,, \\[2ex]
h_c''(\rho)\,=\,\dfrac{P_c'(\rho)}{\rho}\,=\,\rho^{-\g_c-2}\qquad\qquad\mbox{ and }\qquad\qquad
h_c(1)\,=\,h_c'(1)\,=\,0\,,
\end{cases}
$$
and let us define the classical energy 
\begin{equation}  \label{eq:def_E_2} 
E[\rho,u](t) \;:=\;\int_\Omega\left(\dfrac{1}{{\rm Fr}^2}\,h(\rho)\,+\,\dfrac{1}{{\rm Fr}^2}\,h_c(\rho)\,+\,
\dfrac{1}{2}\,\rho\,|u|^2\,+\,
\dfrac{1}{2\,{\rm We}}\,|\nabla\rho|^2\right)dx\,,
\end{equation}
and the BD entropy function
\begin{equation} \label{eq:def_F}
F[\rho](t)\;:=\;\frac{\nu^2}{2}\int_\Omega\rho\,|\nabla\log\rho|^2\,dx\;=\;
2\,\nu^2\int_\Omega\left|\nabla\sqrt{\rho}\right|^2\,dx\,.
\end{equation}
Finally, let us denote by $E[\rho_0,u_0]\,\equiv\,E[\rho,u](0)$ and by
$F[\rho_0]\,\equiv\,F[\rho](0)$ the same energies, when computed on the initial data $\bigl(\rho_0,u_0\bigr)$.

We now give the definition of weak solution to our system. The integrability properties we require (see points (i)
and (ii) below), as well as conditions \eqref{eq:initial} for the initial data, will be justified by energy estimates,
which we will establish in Subsection \ref{ss:energy}.
\begin{defin} \label{d:weak}
Fix initial data $(\rho_0,u_0)$ which satisfy the conditions in \eqref{eq:initial}, with $\rho_0\geq0$.

We say that $\bigl(\rho,u\bigr)$ is a \emph{weak solution} to system \eqref{eq:NSK}-\eqref{eq:bc}
in $[0,T[\,\times\Omega$ (for some $T>0$) with initial datum $(\rho_0,u_0)$ if the following conditions are verified:
\begin{itemize}
 \item[(i)] $\rho\geq0$ almost everywhere, and one has the  properties
 ${\rm Fr}^{-1}\bigl(\rho-1\bigr)\,\in\,L^\infty\bigl([0,T[\,;L^\g(\Omega)\bigr)$,
${\rm Fr}^{-1}\bigl(1/\rho-1\bigr)\,\in\,L^\infty\bigl([0,T[\,;L^{\g_c}(\Omega)\bigr)$, ${\rm We}^{-1/2}\nabla\rho$
and $\nabla\sqrt{\rho}\;\in L^\infty\bigl([0,T[\,;L^2(\Omega)\bigr)$ and
${\rm We}^{-1/2}\nabla^2\rho\in L^2\bigl([0,T[\,;L^2(\Omega)\bigr)$;
\item[(ii)] $\sqrt{\rho}\,u\,\in L^\infty\bigl([0,T[\,;L^2(\Omega)\bigr)$ and $\sqrt{\rho}\,Du\,\in L^2\bigl([0,T[\,;L^2(\Omega)\bigr)$;
\item[(iii)] the mass and momentum equations are satisfied in the weak sense: for any scalar function
$\phi\in\mc{D}\bigl([0,T[\,\times\Omega\bigr)$ one has the equality
$$
-\int^T_0\int_\Omega\biggl(\rho\,\d_t\phi\,+\,\rho\,u\,\cdot\,\nabla\phi\biggr)\,dx\,dt\,=\,\int_\Omega\rho_0\,\phi(0)\,dx\,,
$$
and for any vector-field $\psi\in\mc{D}\bigl([0,T[\times\Omega;\R^3\bigr)$ one has
\begin{eqnarray}
& & \hspace{-0.5cm}
\int_\Omega\rho_0\,u_0\cdot\psi(0)\,dx\,=\,\int^T_0\!\!\int_\Omega\biggl(-\rho\,u\cdot\d_t\psi\,-\,\rho\,u\otimes u:\nabla\psi\,-\,
\frac{1}{{\rm Fr}^2}\,\Pi(\rho)\,\div\psi\,+ \label{eq:weak-momentum} \\
& & \quad +\,\nu\,\rho\,Du:\nabla\psi\,+\,\frac{1}{\rm We}\,\rho\,\Delta\rho\,\div\psi\,+\,
\frac{1}{\rm We}\,\Delta\rho\,\nabla\rho\cdot\psi\,+\,\frac{\mf{c}(x^h)}{\rm Ro}\,e^3\times\rho\,u\cdot\psi\biggr)\,dx\,dt\,. \nonumber
\end{eqnarray}
\end{itemize}
\end{defin}

For such initial data, the existence of weak solutions to system \eqref{eq:NSK} is guaranteed for any fixed value of the positive
parameters ${\rm Fr}$, ${\rm We}$ and ${\rm Ro}$.
\begin{thm} \label{t:weak}
Let $\g_c=2$ in \eqref{eq:def_cold} and $\mf c\,\in\,W^{1,\infty}(\R^2)$ in \eqref{eq:coriolis}.
Fix the value of the Froude, Weber and Rossby numbers, and consider an initial datum $(\rho_0,u_0)$
satisfying conditions \eqref{eq:initial}, with $\rho_0\geq0$.

Then, there exits a global in time weak solution $(\rho,u)$ to system \eqref{eq:NSK}, 
related to $(\rho_0,u_0)$. 
\end{thm}

The previous result can be established arguing exactly as in \cite{B-D-Zat}, so we omit its proof. Actually, the result
of \cite{B-D-Zat} holds true under more general assumptions than ours (as for the cold component of the pressure and the viscosity
coefficient, for instance).

\begin{rem} \label{r:existence}
\begin{itemize}
 \item The hypothesis $\g_c=2$ is assumed just for simplicity here, but, as remarked above, it is not really necessary for existence.
 \item The condition $\mf{c}\,\in\,W^{1,\infty}$ is important in order to take advantage of the BD entropy structure of our system,
see Paragraph \ref{sss:BD}. However, it can be deeply relaxed at this level: see also the discussion at the beginning
of Subsection \ref{ss:regularity}.
\end{itemize}
\end{rem}

\subsection{The singular perturbation problem}

We get now interested in a singular perturbation problem for system \eqref{eq:NSK}. Namely, we want to study the incompressible
and high rotation limit simultaneously, both in the regimes of constant and vanishing capillarity (in the same spirit
of the analysis of \cite{F}).

For doing this, we consider a small parameter $\veps\in\,]0,1]$: we set ${\rm Fr}\,=\,{\rm Ro}\,=\,\veps$,
${\rm We}\,=\,\veps^{2(1-\alpha)}$, for some $0\leq\alpha\leq1$. Notice that the constant capillarity regime corresponds
to the choice $\alpha=0$, while in the other cases we are letting also the capillarity coefficient go to $0$.

Therefore, we end up with the equations
\begin{equation} \label{eq:NSK-sing}
\begin{cases}
\d_t\rho+\div\left(\rho u\right)\,=\,0 \\[1ex]
\d_t\left(\rho u\right)+\div\bigl(\rho u\otimes u\bigr)+\dfrac{1}{\veps^2}\,\nabla\Pi(\rho)+
\dfrac{1}{\veps}\,\mf{C}(\rho,u)-\nu\div\bigl(\rho Du\bigr)-
\dfrac{1}{\veps^{2(1-\alpha)}}\,\rho \nabla\Delta\rho\,=\,0\,,
\end{cases}
\end{equation}

Here we will consider the general instance of \emph{ill-prepared} initial data
$\bigl(\rho,u\bigr)_{|t=0}=\bigl(\rho_{0,\veps},u_{0,\veps}\bigr)$. Namely, we will suppose
the following assumptions on the family $\bigl(\rho_{0,\veps}\,,\,u_{0,\veps}\bigr)_{\veps>0}$:
\begin{itemize}
\item[(i)] $\rho_{0,\veps}\,=\,1\,+\,\veps\,r_{0,\veps}$, with
$\bigl(r_{0,\veps}\bigr)_\veps\,\subset\,H^1(\Omega)\cap L^\infty(\Omega)$ bounded;
\item[(ii)] $1/\rho_{0,\veps}\,=\,1\,+\,\veps\,a_{0,\veps}$, with
$\bigl(a_{0,\veps}\bigr)_\veps\,\subset\,L^2(\Omega)$ bounded;
\item[(iii)] $\bigl(u_{0,\veps}\bigr)_\veps\,\subset\,L^2(\Omega)$ bounded.
\end{itemize}
Up to extraction of a subsequence, we can suppose that
\begin{equation} \label{eq:conv-initial}
r_{0,\veps}\,\rightharpoonup\,r_0\quad\mbox{ in }\;H^1(\Omega)\;,\qquad
a_{0,\veps}\,\rightharpoonup\,a_0\,=\,-\,r_0\quad\mbox{ in }\;L^2(\Omega)\;,\qquad
u_{0,\veps}\,\rightharpoonup\,u_0\quad\mbox{ in }\;L^2(\Omega)\,,
\end{equation}
where we denoted by $\rightharpoonup$ the weak convergence in the respective spaces.

\begin{rem} \label{r:cold-density}
The property $a_0=-r_0$ immediately follows from the weak convergence. Indeed, for any test function $\vphi\in\mc{D}(\Omega)$,
by definition of $a_{0,\veps}$ we have
$$
\int_\Omega a_{0,\veps}\,\vphi\,dx\,=\,-\,\frac{1}{\veps}\int_\Omega\frac{1}{\rho_{0,\veps}}\left(\rho_{0,\veps}-1\right)\,\vphi\,dx\,=\,
-\int_\Omega r_{0,\veps}\,\vphi\,dx\,-\,\veps\int_\Omega r_{0,\veps}\,a_{0,\veps}\,\vphi\,dx\,.
$$
The left-hand side of the previous equality converges to $\int_\Omega a_0\vphi\,dx$; as for the right-hand side, the former term converges
to $-\int_\Omega r_0\vphi\,dx$, and the latter goes to $0$ for $\veps\ra0$.
\end{rem}

\begin{rem} \label{rem:density}
Notice that the choice of constant density profile in the limit, i.e. $\oline{\rho}\equiv1$, is consistent with the balance of forces
acting on the system. As a matter of fact, $\oline{\rho}$ is identified as a solution of the static problem
$$
\nabla\Pi(\oline{\rho})\,=\,-\,\veps^{2\alpha}\,\oline{\rho}\,\nabla\Delta\oline{\rho}\,,
$$
and this relation implies, up to an additive constant,
$$
\qquad\qquad -\,\veps^{2\alpha}\,\Delta\oline{\rho}\,=\,\Xi(\oline{\rho})\,,\qquad\qquad\mbox{ with }\qquad
\Xi(\rho)\,:=\,\int_1^\rho\bigl(\Pi'(\sigma)/\sigma\bigr)\,d\sigma\,.
$$
Of course, $\oline{\rho}\equiv1$ solves the previous elliptic equation. However, depending on the non-linearity $\Pi$, when $\alpha=0$
one is led to consider also non-constant limit density profiles: this was already the case in \cite{F-G-GV-N} for a barotropic Navier-Stokes
system with Coriolis term, when the centrifugal force is assumed of the same order as the pressure and the rotation.

The analogue for capillary fluids will be matter of future studies.
\end{rem}

Furthermore, we need to slightly modify the definition of weak solutions: namely, in addition to the conditions of Definition \ref{d:weak},
we also demand that the weak solutions are constructed in such a way to satisfy relevant uniform bounds in $\veps$.
More precisely, we set $E_\veps[\rho,u]$ and $F_\veps[\rho]$ the energies defined in \eqref{eq:def_E_2} and \eqref{eq:def_F}
respectively, where we take the scaling ${\rm Fr}={\rm Ro}=\veps$, ${\rm We}=\veps^{2(1-\alpha)}$: hence we require that,
for almost every $t\in\,]0,T]$, the following inequalities hold true:
\begin{eqnarray} 
E_\veps[\rho,u](t)\,+\,\nu\int^t_0\int_\Omega\rho\,\left|Du\right|^2\,dx\,d\tau & \leq & E_\veps[\rho_0,u_0] \label{en_est:E} \\
F_\veps[\rho](t)\,+\,\dfrac{\nu}{\veps^2}\int^t_0\int_\Omega P'(\rho)\,|\nabla\sqrt{\rho}|^2\,dx\,d\tau\,+\,
\dfrac{\nu}{\veps^{2(1-\alpha)}}\int^t_0\int_\Omega\left|\nabla^2\rho\right|^2\,dx\,d\tau & \leq & C\,(1+T)\,, \label{en_est:F}
\end{eqnarray}
where the constant $C$ depends just on the triplet $\bigl(E_\veps[\rho_{0},u_0],F_\veps[\rho_{0}],\nu\bigr)$.

From now on, we will focus only on the regime of constant capillarity, i.e. $\alpha=0$. Indeed, our main goal is to capture the effects
of the surface tension in the asymptotics, which seems to be a new feature in this kind of studies. However, we remark that the vanishing capillarity limit (i.e.
$\alpha\in\,]0,1]$) can be dealt with as in \cite{F} when $\mf{c}\equiv1$, or by similar arguments as in Section \ref{s:singular}
when $\mf{c}$ is non-constant.

First of all, let us consider the case of constant rotation axis, namely when $\mf{c}\equiv1$ in \eqref{eq:coriolis}.
\begin{thm} \label{t:sing}
Let $1<\g\leq2$ in \eqref{eq:def_P}, $\alpha=0$ and $\mf{C}(\rho,u)=e^3\times\rho\,u$ in \eqref{eq:NSK-sing}.

Let $\bigl(\rho_{0,\veps},u_{0,\veps}\bigr)_\veps$ be initial data satisfying the hypotheses ${\rm (i)-(ii)-(iii)}$ and
\eqref{eq:conv-initial}, and let $\bigl(\rho_\veps\,,\,u_\veps\bigr)_\veps$ be a family of corresponding weak solutions to system
\eqref{eq:NSK-sing}-\eqref{eq:bc} in $[0,T]\times\Omega$, in the sense of Definition \ref{d:weak}. Suppose that inequalities \eqref{en_est:E}-\eqref{en_est:F}
hold true and that the symmetriy properties of Remark \ref{r:period-bc} are verified. We also define $r_\veps\,:=\,\veps^{-1}\left(\rho_\veps-1\right)$.

Then, up to the extraction of a subsequence, $\bigl(r_\veps\,,\,u_\veps\bigr)_\veps$ weakly converges (in suitable energy
spaces\footnote{See points $(a)$ and $(b)$ in Theorem \ref{t:sing_var}.}) to $(r,u)$,
where $r=r(x^h)$ and $u=\bigl(u^h(x^h),0\bigr)$ are linked by the relation
$u^h\,=\,\nabla^\perp_h\left(\Id-\Delta_h\right)r$. Moreover,
$r$ is a weak solution of the modified quasi-geostrophic equation
\begin{equation} \label{eq:q-geo_0}
\d_t\bigl(\Id-\Delta_h+\Delta_h^2\bigr)r\,+\,\nabla^\perp_h\bigl(\Id-\Delta_h\bigr)r\,\cdot\,\nabla_h\Delta_h^2r\,+\,
\frac{\nu}{2}\,\Delta_h^2\bigl(\Id-\Delta_h\bigr)r\,=\,0
\end{equation}
supplemented with the initial condition $r_{|t=0}\,=\,\wtilde{r}_0$, where $\wtilde{r}_0\,\in\,H^3(\R^2)$ is identified by
$$
\bigl(\Id-\Delta_h+\Delta_h^2\bigr)\,\wtilde{r}_0\,=\,\int_0^1\bigl(\omega^3_0\,+\,r_0\bigr)\,dx^3\,.
$$
\end{thm}
We omit the proof of this result here, because it goes along the lines of the one given in \cite{F}. Its main ingredients are the spectral analysis of the
(constant coefficient) singular perturbation operator and an application of the RAGE Theorem. We remark that the RAGE Theorem allows to deduce strong
convergence properties in suitable norms, and this is the key to pass to the limit in the non-linear terms.

Let us now consider the case of effectively variable rotation axis.
For technical reason, analogously to what done in \cite{G-SR_2006}, we need to assume that the function $\mf c$ has non-degenerate
critical points: namely, we will suppose
\begin{equation} \label{eq:non-crit}
\lim_{\delta\ra0}\;\mc{L}\left(\left\{x^h\,\in\,\R^2\;\Bigl|\;\bigl|\nabla_h\mf{c}(x^h)\bigr|\,\leq\,\delta\right\}\right)\,=\,0\,,
\end{equation}
where we denoted by $\mc L(\mc O)$ the $2$-dimensional Lebesgue measure of a set $\mc O\,\subset\R^2$.

Also, we will require that the gradient of $\mf{c}$ is $\mu$-continuous, for some admissible modulus of continuity $\mu$:
we will recall the precise definition in Appendix \ref{app:continuity}.

For notation convenience, let us also introduce the operator
$$
\mf{D}_{\mf{c}}(f)\,:=\,D_h\bigl(\mf{c}^{-1}\,\nabla_h^\perp f\bigr)\,=\,\frac{1}{2}\,\left(\nabla_h\,+\,^t\nabla_h\right)\bigl(\mf{c}^{-1}\,\nabla_h^\perp f\bigr)
$$
for any scalar function $f\,=\,f(x^h)$.

\begin{thm} \label{t:sing_var}
Let $1<\g\leq2$ in \eqref{eq:def_P}, $\alpha=0$ and $\mf{C}(\rho,u)=\mf{c}(x^h)\,e^3\times\rho\,u$ in \eqref{eq:NSK-sing},
where $\mf{c}\in W^{1,\infty}(\R^2)$ is $\neq0$ almost everywhere and it verifies the non-degeneracy condition \eqref{eq:non-crit}.
Let us also assume that $\nabla_h\mf{c}\,\in\,\mc{C}_\mu(\R^2)$, for some admissible modulus of continuity $\mu$.

Let $\bigl(\rho_{0,\veps},u_{0,\veps}\bigr)_\veps$ be initial data satisfying the hypotheses ${\rm (i)-(ii)-(iii)}$ and
\eqref{eq:conv-initial}, and let $\bigl(\rho_\veps\,,\,u_\veps\bigr)_\veps$ be a family of corresponding weak solutions to system
\eqref{eq:NSK-sing}-\eqref{eq:bc} in $[0,T]\times\Omega$, in the sense of Definition \ref{d:weak}. Suppose that inequalities \eqref{en_est:E}-\eqref{en_est:F}
hold true and that the symmetriy properties of Remark \ref{r:period-bc} are verified.
Define $r_\veps\,:=\,\veps^{-1}\left(\rho_\veps-1\right)$ as above.

Then, up to the extraction of a subsequence, one has the following convergence properties: 
\begin{itemize}
\item[(a)] $r_\veps\,\rightharpoonup\,r$ in $L^\infty\bigl([0,T];H^1(\Omega)\bigr)\,\cap\,L^2\bigl([0,T];H^2(\Omega)\bigr)$,
\item[(b)] $\sqrt{\rho_\veps}\,u_\veps\,\rightharpoonup\,u$ in $L^\infty\bigl([0,T];L^2(\Omega)\bigr)$ and
 $\sqrt{\rho_\veps}\,Du_\veps\,\rightharpoonup\,Du$ in $L^2\bigl([0,T];L^2(\Omega)\bigr)$,
\end{itemize}
where, this time, $r=r(x^h)$ and $u=\bigl(u^h(x^h),0\bigr)$ verify the relation
$\mf{c}(x^h)\,u^h\,=\,\nabla^\perp_h\left(\Id-\Delta_h\right)r$. Moreover,
$r$ solves (in the weak sense) the equation
\begin{equation} \label{eq:lim_var}
\d_t\left(r\,-\,\div_{\!h}\!\left(\frac{1}{\mf{c}^2}\,\nabla_h\bigl(\Id-\Delta_h\bigr)r\right)\right)\,+\,
\nu\,\;^t\mf{D}_{\mf{c}}\,\circ\,\mf{D}_{\mf{c}}\bigl((\Id-\Delta_h)r\bigr)\,=\,0
\end{equation}
supplemented with the initial condition $r_{|t=0}\,=\,\wtilde{r}_0$, where $\wtilde{r}_0$ is defined by
$$
\wtilde{r}_0\,-\,\div_{\!h}\!\left(\frac{1}{\mf{c}^2}\,\nabla_h\bigl(\Id-\Delta_h\bigr)\wtilde{r}_0\right)\,=\,
\int_0^1\Bigl({\rm curl}_h\bigl(\mf{c}^{-1}\,u^h_0\bigr)\,+\,r_0\Bigr)\,dx^3\,.
$$
\end{thm}

Notice that we have the identity
$$
\,^t\mf{D}_{\mf{c}}\,\circ\,\mf{D}_{\mf{c}}(f)\,=\,
\nabla_h^\perp\cdot\left(\frac{1}{\mf{c}}\,\nabla_h\cdot\mf{D}_{\mf{c}}(f)\right)\,,
$$
where we used the notations $\div f$ and $\nabla\cdot f$ in an equivalent way.
We also remark here that, for $\mf{c}\equiv1$, this operator reduces to $(1/2)\Delta_h^2f$.

\begin{rem} \label{r:var_axis}
As already pointed out in \cite{G-SR_2006} (see also \cite{F-G-GV-N}), the limit equation is linear in the case of variable rotation
axis. Indeed, the limit motion is much more constrained in this situation, and correspondingly the kernel of the singular perturbation operator is smaller than in the
instance of constant axis.

Also, notice that having non-constant $\mf{c}$ makes variable coefficients appear in the limit equation.
\end{rem}

\section{A priori estimates} \label{s:bounds}

The present section is devoted to show uniform bounds for smooth solutions of our system, written in the general form \eqref{eq:NSK}.
Throughout this section we will suppose ${\rm Fr}$, ${\rm We}$ and ${\rm Ro}$ to be fixed. However, we will keep track of them
in the first paragraph: this is relevant in view of the analysis of Section \ref{s:singular}. 

Besides, the next computations will give complete justification to a technical assumption in \cite{F}
(see Remark \ref{r:anisotropic} below).

\subsection{Energy estimates} \label{ss:energy}

Suppose that $(\rho,u)$ is a smooth solution to system \eqref{eq:NSK} in $\R_+\times\Omega$,
related to the smooth initial datum $\bigl(\rho_0,u_0\bigr)$. We establish here energy estimates for $(\rho,u)$.

\begin{rem} \label{r:uniform}
The construction given in e.g. \cite{B-D-Zat} ensures the existence of smooth approximated solutions, which converge to a weak solution of our original system
and which are compatible with the uniform bounds given by the BD entropy structure of the equations.

Therefore, the estimates established here will be inherited by the family of weak solutions we are going to consider, see Subsection \ref{ss:uniform}.
\end{rem}

\subsubsection{Classical energy}

First of all, let us show the classical energy conservation.

\begin{prop} \label{p:E}
Let $(\rho,u)$ be a smooth solution to system \eqref{eq:NSK} in $\R_+\times\Omega$, with
initial datum $\bigl(\rho_0,u_0\bigr)$. Then, for all $t\in\R_+$, one has
$$
\frac{d}{dt}E[\rho,u](t)\,+\,\nu\int_\Omega\rho\,|Du|^2\,dx\,=\,0\,.
$$
\end{prop}

\begin{proof}
First of all, we multiply the second relation in system \eqref{eq:NSK} by $u$: by use of the mass equation and the fact that 
$\mf{C}(\rho,u)$ is orthogonal to $u$, we arrive at the equality
$$ 
\frac{1}{2}\frac{d}{dt}\!\int_\Omega\left(\rho\,|u|^2+\frac{1}{{\rm We}}\,|\nabla\rho|^2\right)dx\,+\,
\frac{1}{{\rm Fr}^2}\int_\Omega\bigl(P'(\rho)+P_c'(\rho)\bigr)\nabla\rho\cdot u\,dx\,+\,
\nu\!\int_\Omega\rho\,Du:\nabla u\,dx\,=\,0\,.
$$ 
On the one hand, we have the identity $Du:\nabla u=|Du|^2$; on the other hand, multiplying the equation for
$\rho$ by $h'(\rho)/{\rm Fr}^2$ gives
$$
\frac{1}{{\rm Fr}^2}\int_\Omega P'(\rho)\,\nabla\rho\cdot u\,dx\,=\,
\frac{1}{{\rm Fr}^2}\,\frac{d}{dt}\int_\Omega h(\rho)\,dx\,.
$$
Notice that an analogous equality holds true also for the cold part $P_c$ of the pressure. 

Putting the obtained relations into the previous one concludes the proof of the proposition.
\end{proof}

From the previous energy conservation, we infer the following bounds.
\begin{coroll} \label{c:E}
Let $\g_c=2$. Let $(\rho,u)$ be a smooth solution to system \eqref{eq:NSK} in $\R_+\times\Omega$, with
initial datum $\bigl(\rho_0,u_0\bigr)$, and assume that $E[\rho_0,u_0]<+\infty$.
Then one has the following properties:
$$
\sqrt{\rho}\,u\;,\;\frac{1}{\sqrt{\rm We}}\,\nabla\rho\quad\in\;L^\infty\bigl(\R_+;L^2(\Omega)\bigr)\qquad\mbox{ and }\qquad
\sqrt{\rho}\,Du\;\in\;L^2\bigl(\R_+;L^2(\Omega)\bigr)\,.
$$
Moreover, we also get
$$
\frac{1}{\rm Fr}\left(\rho\,-\,1\right)\,\in\,L^\infty\bigl(\R_+;L^\g(\Omega)\bigr)\qquad\mbox{ and }\qquad
\frac{1}{\rm Fr}\left(\frac{1}{\rho}\,-\,1\right)\,\in\,L^\infty\bigl(\R_+;L^2(\Omega)\bigr)\,.
$$
\end{coroll}

\begin{proof}
The first properties are immediate, after integrating the relation of Proposition \ref{p:E} in time. So, let us focus
on the density term.

For $1<\g\leq2$, Lemma 2 of \cite{J-L-W}, combined with Proposition \ref{p:E}, tells us
$$
\left\|\rho-1\right\|^\g_{L^\infty_T(L^\g)}\,\leq\,C\left({\rm Fr}^\g\,+\,{\rm Fr}^2\right)\,,
$$
which immediately gives the first property.

Let us consider negative powers of the density: for $\g_c=2$, it is easy to see that
$$
H_c(\rho)\,:=\,\g_c\,(\g_c+1)\,h_c(\rho)\,-\,\left|\frac{1}{\rho}\,-\,1\right|^{\g_c}\,\geq\,0
$$
for all $\rho\geq0$, and this concludes the proof of the corollary.
\end{proof}

\subsubsection{BD entropy} \label{sss:BD}

We want now to take advantage of the BD entropy structure of our system. Let us start with a lemma, which is the analogue of
Proposition 3.3 in \cite{F}. For the sake of completeness, we give the complete proof in Appendix \ref{app:BD}.
\begin{lemma} \label{l:F}
Let $(\rho,u)$ be a smooth solution to system \eqref{eq:NSK} in $\R_+\times\Omega$, related to the initial datum $\bigl(\rho_0,u_0\bigr)$.
Then there exists a ``universal'' constant $C>0$ such that, for all $t\in\R_+$, one has
\begin{eqnarray}
& & \hspace{-0.3cm}
\frac{1}{2}\!\!\int_\Omega\rho\,\left|u\,+\,\nu\,\nabla\log\rho\right|^2\,dx\,+\,
\frac{\nu}{{\rm We}}\!\!\int^t_0\!\!\int_\Omega\left|\nabla^2\rho\right|^2\,dx\,d\tau\,+ 
\,\frac{4\nu}{{\rm Fr}^2}\!\!\int^t_0\!\!\int_\Omega\Pi'(\rho)\,\left|\nabla\sqrt{\rho}\right|^2\,dx\,d\tau\,\leq \label{est:F} \\
& & \qquad\qquad\qquad\qquad\qquad\qquad
\leq\,C\bigl(F[\rho_0]+E[\rho_0,u_0]\bigr)\,+\,
\frac{\nu}{\rm Ro}\left|\int^t_0\!\!\int_\Omega\mf{c}(x^h)\,e^3\times u\cdot\nabla\rho\,dx\,d\tau\right|.  \nonumber
\end{eqnarray}
\end{lemma}

Let us point out here that the last term on the left-hand side of the previous relation can be also written as
\begin{eqnarray}
\frac{4\,\nu}{{\rm Fr}^2}\int_\Omega\bigl(P'(\rho)+P'_c(\rho)\bigr)\left|\nabla\sqrt{\rho}\right|^2 & = & 
\frac{\nu}{{\rm Fr}^2}\int_\Omega\left(\rho^{\g-2}\,+\,\rho^{-\g_c-2}\right)\left|\nabla\rho\right|^2 \label{eq:P+P_c} \\
& = & \frac{C_\g\,\nu}{{\rm Fr}^2}\int_\Omega\left|\nabla\left(\rho^{\g/2}\right)\right|^2\,+\,
\frac{C_{\g_c}\,\nu}{{\rm Fr}^2}\int_\Omega\left|\nabla\left(\rho^{-\g_c/2}\right)\right|^2\,, \nonumber
\end{eqnarray}
for some positive constants $C_\g$ and $C_{\g_c}$. In particular, in our case $\g_c=2$, $C_{\g_c}=1$.

Next, let us give now some estimates for the density (again, see the proof in Appendix \ref{app:BD}).
\begin{lemma} \label{l:rho_BD}
There exists a ``universal'' constant $C>0$ such that
$$
\|\rho-1\|_{L^\infty_t(L^2)}\,\leq\,C\left({\rm Fr}\,+\,\bigl(1-\mathds{1}_2(\g)\bigr)\,\sqrt{\rm We}\right)\,,
$$
where we have set $\mathds{1}_2(\g)=1$ if $\g=2$ and $\mathds{1}_2(\g)=0$ otherwise.

Moreover, for any $0<\delta\leq1/2$ and any $1\leq p\leq4/(1+2\delta)$ one has
\begin{eqnarray*}
\hspace{-1.5cm}
\|\rho-1\|^p_{L^p_t(L^\infty)} & \leq & C_p\biggl(\left({\rm Fr}+\bigl(1-\mathds{1}_2(\g)\bigr)\sqrt{\rm We}\right)\,t\,+ \\
& & \qquad\qquad +\,({\rm We})^{p(3-2\delta)/4}\;\nu^{-1/q}\;t^{1-1/q}\,
\left(\frac{\nu}{\rm We}\,\left\|\nabla^2\rho\right\|^2_{L^2_t(L^2)}\right)^{1/q}\biggr)\,,
\end{eqnarray*}
where we have defined $q\,:=\,4/\bigl((1+2\delta)p\bigr)\,\in\,[1,4/(1+2\delta)]$,
and the constant $C_p$ depends also on the value of $p$.
\end{lemma}

We are now ready to control the Coriolis term in the estimates of Lemma \ref{l:F}. We will give two different
types of estimates: the first one is useful in order to justify Definition \ref{d:weak} of weak solutions.
The second kind of inequalities, instead, will be relevant in studying the singular limit problem (see Section
\ref{s:singular}). Again, the proof is postponed to Appendix \ref{app:BD}.

\begin{lemma} \label{l:rot}
\begin{itemize}
\item[(i)] There exists a positive constant $C$, just depending on $E[\rho_0,u_0]$, such that
$$
\frac{\nu}{\rm Ro}\left|\int^t_0\!\!\int_\Omega\mf{c}(x^h)\,e^3\times u\cdot\nabla\rho\,dx\,d\tau\right|\,\leq\,
\frac{C}{\rm Ro}\,\int^t_0\bigl(F[\rho](\tau)\bigr)^{1/2}\,d\tau\,.
$$
\item[(ii)] Moreover, one has the following estimate: for any $1<\g\leq2$,
\begin{eqnarray*}
& & \hspace{-1cm}
\frac{\nu}{\rm Ro}\left|\int^t_0\!\!\int_\Omega\mf{c}(x^h)\,e^3\times u\cdot\nabla\rho\,dx\,d\tau\right|\;\leq\;
C\,\nu\,(1+t)\;\frac{{\rm Fr}+\sqrt{{\rm We}}}{\rm Ro}\,\left(1+{\rm Fr}+\sqrt{{\rm We}}\right)^{\!1/2}\,+ \\
& & \qquad\qquad\qquad\qquad\qquad
+\,C\,\nu\,(1+t)\,({\rm We})^{3/5}\,\left(\frac{{\rm Fr}+\sqrt{{\rm We}}}{\rm Ro}\right)^{\!8/5}\,+\,
\frac{3}{4}\frac{\nu}{{\rm We}}\,\left\|\nabla^2\rho\right\|^2_{L^2_t(L^2)}\,.
\end{eqnarray*}
Alternatively, in the particular instance $\g=2$, one
can get the different bound
\begin{eqnarray*}
& & \hspace{-1cm}
\frac{\nu}{\rm Ro}\left|\int^t_0\!\!\int_\Omega\mf{c}(x^h)\,e^3\times u\cdot\nabla\rho\,dx\,d\tau\right|\;\leq\;
C\,\nu\,t\,\left(\frac{{\rm Fr}}{{\rm Ro}}\right)^{\!2}\,+\,C\,\nu\,\sqrt{t}\,\frac{\rm Fr}{\rm Ro}\,\bigl(1+{\rm Fr}\bigr)^{\!1/2}\,+ \\
& & \qquad\qquad +\,C\,\nu\,(1+t)\,({\rm We})^{3/5}\,\left(\frac{{\rm Fr}}{\rm Ro}\right)^{\!8/5}\,+\,
\frac{1}{2}\,\frac{\nu}{{\rm Fr}^2}\,\left\|\nabla\rho\right\|^2_{L^2_t(L^2)}\,+\,
\frac{3}{4}\,\frac{\nu}{\rm We}\,\left\|\nabla^2\rho\right\|^2_{L^2_t(L^2)}\,.
\end{eqnarray*}
\end{itemize}
\end{lemma}

By the previous lemmas, we easily infer the BD entropy estimate for our system.
\begin{prop} \label{p:F}
Let $(\rho,u)$ be a smooth solution to system \eqref{eq:NSK} in $\R_+\times\Omega$, related to the initial
datum $\bigl(\rho_0,u_0\bigr)$.

\begin{itemize}
\item[(i)] There exists a constant $C>0$, just depending on $E[\rho_0,u_0]$ and $F[\rho_0]$, such that, for all $T\in\R_+$
fixed, one has
$$
\hspace{-0.5cm}
\sup_{[0,T]}F[\rho]\,+\,\frac{\nu}{{\rm We}}\int^T_0\!\!\int_\Omega\left|\nabla^2\rho\right|^2\,dx\,dt\,+\,
\frac{\nu}{{\rm Fr}^2}\int^T_0\!\!\int_\Omega\Pi'(\rho)\,\left|\nabla\sqrt{\rho}\right|^2\,dx\,dt\,\leq\,
C+C\left(\frac{T}{\rm Ro}\right)^{\!\!2}\,.
$$
\item[(ii)] Alternatively, for any $1<\g\leq2$ and any $T\in\R_+$ fixed, we have the estimate
\begin{eqnarray*}
& & \hspace{-2cm} \sup_{[0,T]}F[\rho]\,+\,\frac{\nu}{{\rm We}}\int^T_0\!\!\int_\Omega\left|\nabla^2\rho\right|^2\,dx\,dt\,+\,
\frac{\nu}{{\rm Fr}^2}\int^T_0\!\!\int_\Omega\Pi'(\rho)\,\left|\nabla\sqrt{\rho}\right|^2\,dx\,dt\,\leq \\
& & \qquad\qquad\qquad\qquad\qquad\qquad\qquad
\leq\,C\,+\,C_1\,\nu\,(1+T)\,\Theta\,+\,C_2\,\nu\,(1+T)\,\Theta^{8/5}\,,
\end{eqnarray*}
where $\Theta\,=\,\Theta({\rm Fr},{\rm Ro},{\rm We})\,=\,\left({\rm Fr}+\sqrt{{\rm We}}\right)\!/\,{\rm Ro}$ 
and the positive constants $C_1$ and $C_2$ are uniformly bounded for $({\rm Fr},{\rm We})$ varying in some compact set
$[0,\oline{\rm Fr}]\times[0,\oline{\rm We}]$.

In the case $\g=2$, one can derive also an analogous inequality, with the right-hand side replaced by the quantity
$$
C\,+\,C\,\nu\,T\,\wtilde{\Theta}^2\,+\,C_1\,\nu\,\sqrt{T}\,\wtilde{\Theta}\,+\,C_2\,\nu\,(1+T)\,\wtilde{\Theta}^{8/5}\,,
$$
where $C_1$ and $C_2$ have the same property as above, and we have defined $\wtilde{\Theta}\,:=\,{\rm Fr}/{\rm Ro}$.
\end{itemize}
\end{prop}

\begin{proof}
Our starting point is inequality \eqref{est:F}: we estimate from below the first term in the left-hand side by
$$
\frac{1}{2}\int_\Omega\rho(t)\,\left|u(t)\,+\,\nu\,\nabla\log\rho(t)\right|^2\,dx\,\geq\,F[\rho](t)\,-\,
\frac{1}{2}\int_\Omega\rho(t)\,|u(t)|^2\,dx\,,
$$
where the last term can be clearly moved on the right-hand side of \eqref{est:F} and controlled by $E[\rho_0,u_0]$
(recall Proposition \ref{p:E}).

Let us now consider the rotating term, and bound it by the first inequality in Lemma \ref{l:rot}. Notice that, at this point, one could use
Young inequality to get $F^{1/2}\leq1+F$ and then apply Gronwall lemma: this would give an exponential growth in time for $F$ and
the other terms we want to control. So we prefer to argue in a finer way.

More precisely, for any $t\in[0,T]$ we have the bound
$$
\frac{C}{\rm Ro}\,\int^t_0\bigl(F[\rho](\tau)\bigr)^{1/2}\,d\tau\,\leq\,
\frac{C\,t}{\rm Ro}\,\left(\sup_{[0,t]}F[\rho]\right)^{1/2}\,\leq\,\frac{C\,t^2}{{\rm Ro}^2}\,+\,\frac{1}{2}\,\sup_{[0,t]}F[\rho]\,.
$$
and then, from \eqref{est:F}, we find
$$
F[\rho](t)\,+\,\frac{\nu}{{\rm We}}\int^t_0\!\!\int_\Omega\left|\nabla^2\rho\right|^2\,+\,
\frac{4\nu}{{\rm Fr}^2}\int^t_0\!\!\int_\Omega\bigl(P'(\rho)+P_c'(\rho)\bigr)\,\left|\nabla\sqrt{\rho}\right|^2\,\leq\,
C\,+\,\frac{C\,t^2}{{\rm Ro}^2}\,+\,\frac{1}{2}\,\sup_{[0,t]}F[\rho]\,.
$$
Taking first the $\sup_{[0,T]}$ of the member on the right, and then the $\sup_{[0,T]}$ of the member on the left, we deduce item $(i)$
of our statement.

On the other hand, if we bound the Coriolis term in \eqref{est:F} using the second type of estimates of Lemma \ref{l:rot},
for any $1<\g\leq2$ we easily find the former inequality in item $(ii)$.
%

In the special case $\g=2$, we can alternatively use the last bound of Lemma \ref{l:rot}: since
$P'(\rho)|\nabla\sqrt{\rho}|^2\,=\,|\nabla\rho|^2$, we can absorbe the last two terms of the right-hand side into the left-hand term.
Then, the inequality with $\wtilde{\Theta}$ immediately follows.
\end{proof}

\begin{rem} \label{r:anisotropic}
The previous proposition suggests to take $\g=2$ when considering the vanishing capillarity
regime. As a matter of fact, if ${\rm We}=\veps^{2(1-\alpha)}$, for some $0<\alpha\leq1$, uniform bounds in $\veps$
seem to be out of reach without resorting to $\wtilde{\Theta}$ rather than to $\Theta$.
\end{rem}

From Proposition \ref{p:F}, we deduce the next statement. Notice that, here below, the last assertion derives also from Lemma \ref{l:rho_BD}.
\begin{coroll} \label{c:F}
Let $(\rho,u)$ be a smooth solution to system \eqref{eq:NSK} in $\R_+\times\Omega$, with
initial datum $\bigl(\rho_0,u_0\bigr)$, and assume that $E[\rho_0,u_0]$ and $F[\rho_0]$ are finite.
Then one has the following bounds:
$$
\begin{cases}
\nabla\sqrt{\rho}\;\in\;L^\infty_{loc}\bigl(\R_+;L^2(\Omega)\bigr) \\[1ex]
\dfrac{\sqrt{\nu}}{\sqrt{\rm We}}\,\nabla^2\rho\;,\quad\dfrac{\sqrt{\nu}}{\rm Fr}\,\nabla\!\left(\rho^{\g/2}\right)\;,\quad
\dfrac{\sqrt{\nu}}{\rm Fr}\,\nabla\!\left(\dfrac{1}{\rho}\right)\qquad \in\;L^2_{loc}\bigl(\R_+;L^2(\Omega)\bigr)\,.
\end{cases}
$$
In particular, the quantity $\bigl(\rho-1\bigr)/\sqrt{\rm We}$ belongs to $L^2_{loc}\bigl(\R_+;L^\infty(\Omega)\bigr)$.
\end{coroll}

\subsection{Additional bounds} \label{ss:unif-b}

In the present paragraph we show further properties, which can be deduced from the previous controls
of Corollary \ref{c:E} and Corollary \ref{c:F}.

In the sequel, we will still keep track of the dependence of the various quantities on the Froude, Rossby and Weber numbers.
On the contrary, we will drop out the dependence on $\nu$, since for us it will be always a fixed positive parameter.

\medbreak
Hence, let us fix a $T\in\R_+$. First of all, from Corollaries \ref{c:E} and \ref{c:F} above, we get
$$
\frac{1}{\rm Fr}\left(\frac{1}{\rho}\,-\,1\right)\;\in\;L^\infty_T\bigl(L^2(\Omega)\bigr)\,\cap\,L^2_T\bigl(H^1(\Omega)\bigr)\,.
$$
Then, since $\left|1/\sqrt{\rho}\,-\,1\right|\leq\left|1/\rho\,-\,1\right|$, by Sobolev embeddings we also infer
$$
\frac{1}{\rm Fr}\left(\frac{1}{\sqrt{\rho}}\,-\,1\right)\;\in\;L^\infty_T\bigl(L^2(\Omega)\bigr)\,\cap\,L^2_T\bigl(L^6(\Omega)\bigr)\,.
$$

Therefore, thanks to energy estimates of Proposition \ref{p:E}, we deduce first that
\begin{equation} \label{unif-bound:u}
u\,=\,\sqrt{\rho}\,u\,+\,\left(\frac{1}{\sqrt{\rho}}\,-\,1\right)\,\sqrt{\rho}\,u\quad\in\,L^\infty_T\bigl(L^2\bigr)
+L^2_T\bigl(L^{3/2}\bigr)\,\hookrightarrow\,L^2_T\bigl(L^{3/2}_{loc}\bigr)\,,
\end{equation}
and hence, by an analogous decomposition, also that
\begin{equation} \label{unif-bound:Du}
Du\,=\,\sqrt{\rho}\,Du\,+\,\left(\frac{1}{\sqrt{\rho}}\,-\,1\right)\,\sqrt{\rho}\,Du\quad\in\,\Bigl(L^2_T\bigl(L^2\bigr)
+L^1_T\bigl(L^{3/2}\bigr)\Bigr)\,\cap\,\Bigl(L^2_T\bigl(L^2+L^1\bigr)\Bigr)\,.
\end{equation}
Notice that, in particular, $Du$ belongs to $L^1_T\bigl(L^{3/2}_{loc}\bigr)$; therefore, Sobolev embeddings
implies the additional property $u\,\in\,L^1_T\bigl(L^3_{loc}\bigr)$.

We turn now our attention to the quantity $\rho\,u$. Exploiting the same decomposition as above, together with the
$L^2_T\bigl(L^\infty\bigr)$ control on $\rho-1$ provided by the BD entropy conservation (recall Corollary \ref{c:F}), we have
\begin{equation} \label{unif-bound:rho-u}
\rho\,u\,=\,\sqrt{\rho}\,u\,+\,\left(\sqrt{\rho}\,-\,1\right)\,\sqrt{\rho}\,u\quad\in\,L^\infty_T\bigl(L^2+L^{3/2}\bigr)\,\cap\,
\Bigl(L^\infty_T\bigl(L^2\bigr)+L^2_T\bigl(L^2\bigr)\Bigr)\,.
\end{equation}
In particular, this implies that $\rho\,u$ belongs also to $L^\infty_T\bigl(L^{3/2}_{loc}\bigr)\,\cap\,L^2_T\bigl(L^2\bigr)$.
On the other hand, for its gradient we have
$$
D(\rho\,u)\,=\,\rho\,Du\,+\,u\,D\rho\,=\,\sqrt{\rho}\,Du\,+\,\left(\sqrt{\rho}-1\right)\,\sqrt{\rho}\,Du\,+\,\sqrt{\rho}\,u\,D\sqrt{\rho}\,.
$$
The first two terms are in $L^2_T\bigl(L^2+L^{3/2}\bigr)$, while the last one is in $L^\infty_T\bigl(L^1\bigr)$. Then we find
$D(\rho\,u)\,\in\,\Bigl(L^2_T\bigl(L^2+L^{3/2}\bigr)+L^\infty_T\bigl(L^1\bigr)\Bigr)\,\hra\,L^2_T\bigl(L^1_{loc}\bigr)$.

Finally, let us consider the quantity $\rho^{3/2}\,u$: arguing exactly as in \eqref{unif-bound:rho-u}, we get
\begin{equation} \label{unif-bound:rho^32-u}
\rho^{3/2}\,u\,=\,\sqrt{\rho}\,u\,+\,\left(\rho\,-\,1\right)\,\sqrt{\rho}\,u\quad\in\,L^\infty_T\bigl(L^2+L^{3/2}\bigr)\,\cap\,
\Bigl(L^\infty_T\bigl(L^2\bigr)+L^2_T\bigl(L^2\bigr)\Bigr)\,.
\end{equation}
Furthermore, we notice that
\begin{eqnarray}
D\bigl(\rho^{3/2}\,u\bigr) & = & \rho\,\sqrt{\rho}\,Du\,+\,\frac{3}{2}\,\sqrt{\rho}\,u\,D\rho \label{unif-bound:D(rho-u)} \\
& = & \sqrt{\rho}\,Du\,+\,\left(\rho\,-\,1\right)\sqrt{\rho}\,Du\,+\,
\frac{3}{2}\,\sqrt{\rho}\,u\,D\rho\quad\in\,L^2_T\bigl(L^2(\Omega)+L^{3/2}(\Omega)\bigr)\,. \nonumber
\end{eqnarray}
Indeed, the first term in the right-hand side of the last relation clearly belongs to $L^2_T\bigl(L^2\bigr)$; moreover, by Corollaries
\ref{c:E} and \ref{c:F} and Sobolev embeddings, the second and the third terms are in $L^2_T\bigl(L^{3/2}\bigr)$.
Therefore, by use of Proposition \ref{p:emb_hom-besov}, we deduce that
$\rho^{3/2}\,u$ belongs also to $L^2_T\bigl(L^3(\Omega)\bigr)$.

\section{The singular perturbation problem} \label{s:singular}

In the present section we prove Theorem \ref{t:sing_var}.
We first establish uniform bounds for the family $\bigl(\rho_\veps,u_\veps\bigr)_\veps$
of weak solutions we are considering. Then, we study the singular pertubation operator, showing constraints on the limit-points
of this family. Finally, we pass to the limit in the weak formulation of the equations, proving the convergence to
equation \eqref{eq:lim_var}.

\subsection{Uniform estimates} \label{ss:uniform}

By use of the analysis of Section \ref{s:bounds} (keep in mind also what said in Remark \ref{r:uniform}), we establish here uniform bounds for the family of weak solutions
$\bigl(\rho_\veps,u_\veps\bigr)_\veps$. Recall that we have fixed $\g_c=2$, $\alpha=0$ and $1<\g\leq2$.

First of all, we consider the energies $E_\veps[\rho,u]$ and $F_\veps[\rho]$, defined respectively by \eqref{en_est:E} and
\eqref{en_est:F}. We remark that, under our hypotheses on the initial data, we deduce the existence of a ``universal constant'' $C_0>0$
such that
$$
E_\veps[\rho_\veps,u_\veps](0)\,+\,F_\veps[\rho_\veps](0)\,\leq\,C_0\,.
$$

Then, by use of Corollary \ref{c:E} we immediately infer the following properties.
\begin{prop} \label{p:sing_b-E}
Let $\bigl(\rho_\veps,u_\veps\bigr)_\veps$ be the family of weak solutions to system \eqref{eq:NSK-sing}
considered in Theorem \ref{t:sing_var}.
Then it satisfies the following bounds, uniformly in $\veps$:
$$
\sqrt{\rho_\veps}\,u_\veps\;\in\;L^\infty\bigl(\R_+;L^2(\Omega)\bigr)\qquad\mbox{ and }\qquad
\sqrt{\rho_\veps}\,Du_\veps\;\in\;L^2\bigl(\R_+;L^2(\Omega)\bigr)
$$
for the velocity fields, and for the densities
\begin{eqnarray*}
& & \hspace{-2cm}
\dfrac{1}{\veps}\left(\rho_\veps\,-\,1\right)\;\in\;L^\infty\bigl(\R_+;L^\g(\Omega)\bigr)\;,\qquad
\dfrac{1}{\veps}\left(\dfrac{1}{\rho_\veps}\,-\,1\right)\;\in\;L^\infty\bigl(\R_+;L^2(\Omega)\bigr)\;, \\[1ex]
& & \dfrac{1}{\veps}\,\nabla\rho_\veps\;\in\;L^\infty\bigl(\R_+;L^2(\Omega)\bigr)\,.
\end{eqnarray*}
\end{prop}

\begin{rem} \label{r:rho_L^2}
In particular, under our assumptions we always have 
$$ 
\left\|\rho_\veps\,-\,1\right\|_{L^\infty(\R_+;L^2(\Omega))}\,\leq\,C\,\veps\,.
$$ 
\end{rem}

As for the BD entropy structure, Corollary \ref{c:F} implies the following estimates; as for the last sentence, one has to use Lemma
\ref{l:density} in the appendix.
\begin{prop} \label{p:sing_b-BD}
Let $\bigl(\rho_\veps,u_\veps\bigr)_\veps$ be the family of weak solutions to system \eqref{eq:NSK-sing}
considered in Theorem \ref{t:sing_var}.
Then one has the following bounds, uniformly for $\veps>0$:
$$
\begin{cases}
\nabla\sqrt{\rho_\veps}\;\in\;L^\infty_{loc}\bigl(\R_+;L^2(\Omega)\bigr) \\[1ex]
\dfrac{1}{\veps}\,\nabla^2\rho_\veps\;,\quad\dfrac{1}{\veps}\,\nabla\!\left(\rho_\veps^{\g/2}\right)\;,\quad
\dfrac{1}{\veps}\,\nabla\!\left(\dfrac{1}{\rho_\veps}\right)\qquad \in\;L^2_{loc}\bigl(\R_+;L^2(\Omega)\bigr)\,.
\end{cases}
$$
In particular, the family $\bigl(\veps^{-1}\,(\rho_\veps-1)\bigr)_\veps$ is bounded in $L^p_{loc}\bigl(\R_+;L^\infty(\Omega)\bigr)$
for any $2\leq p<4$.
\end{prop}

Moreover, arguing exactly as in Subsection \ref{ss:unif-b}, we can establish also the following bounds, uniformly in $\veps$.
First of all, by the decompositions \eqref{unif-bound:u} and \eqref{unif-bound:Du} we immediately have
\begin{eqnarray*}
\bigl(u_\veps\bigr)_\veps & \subset & L^\infty_T\bigl(L^2\bigr)+L^2_T\bigl(L^{3/2}\bigr)\,\hookrightarrow\,
L^2_T\bigl(L^{3/2}_{loc}\bigr) \\ 
\bigl(Du_\veps\bigr)_\veps & \subset & \Bigl(L^2_T\bigl(L^2\bigr)
+L^1_T\bigl(L^{3/2}\bigr)\Bigr)\,\cap\,\Bigl(L^2_T\bigl(L^2+L^1\bigr)\Bigr)\,. 
\end{eqnarray*}
In particular, $\bigl(Du_\veps\bigr)_\veps$ is uniformly bounded in $L^1_T\bigl(L^{3/2}_{loc}\bigr)$; therefore,
by Sobolev embeddings we gather also the additional continuous inclusion $\bigl(u_\veps\bigr)_\veps\,\subset\,L^1_T\bigl(L^3_{loc}\bigr)$.

Furthermore, we also infer the uniform bounds
\begin{eqnarray}
\bigl(\rho_\veps\,u_\veps\bigr)_\veps & \subset & L^\infty_T\bigl(L^2+L^{3/2}\bigr)\,\cap\,
\Bigl(L^\infty_T\bigl(L^2\bigr)+L^2_T\bigl(L^2\bigr)\Bigr) \nonumber \\
\bigl(D(\rho_\veps\,u_\veps)\bigr)_\veps & \subset & L^2_T\bigl(L^2+L^{3/2}\bigr)\,+\,L^\infty_T\bigl(L^1\bigr)\;\hra\;
L^2_T\bigl(L^1_{loc}\bigr)\,. \label{sing-b:D_rho-u}
\end{eqnarray}
In particular, we deduce that $\bigl(\rho_\veps\,u_\veps\bigr)_\veps$ is a bounded family in 
$L^\infty_T\bigl(L^{3/2}_{loc}\bigr)\,\cap\,L^2_T\bigl(L^2\bigr)$.

For the sake of completeness let us also establish uniform bounds on quantities related to $\rho^{3/2}_\veps\,u_\veps$.
First of all, arguing exactly as in \eqref{unif-bound:rho^32-u}, we get
$$
\bigl(\rho^{3/2}_\veps\,u_\veps\bigr)_\veps\;\subset\;L^\infty_T\bigl(L^2+L^{3/2}\bigr)\,\cap\,
\Bigl(L^\infty_T\bigl(L^2\bigr)+L^2_T\bigl(L^2\bigr)\Bigr)\,;
$$
on the other hand, analogously to \eqref{unif-bound:D(rho-u)}, we have also the uniform embedding
\begin{equation} \label{sing-b:D(rho-u)}
\Bigl(D\bigl(\rho^{3/2}_\veps\,u_\veps\bigr)\Bigr)_\veps\;\subset\;L^2_T\bigl(L^2(\Omega)+L^{3/2}(\Omega)\bigr)\,.
\end{equation}
Therefore, by Proposition \ref{p:emb_hom-besov} we infer that
$\bigl(\rho^{3/2}_\veps\,u_\veps\bigr)_\veps$ is uniformly bounded in $L^2_T\bigl(L^3(\Omega)\bigr)$.

Finally, from this fact combined with the usual decomposition $\sqrt{\rho_\veps}\,=\,1+(\sqrt{\rho_\veps}-1)$
and Sobolev embeddings, it follows also that
$$ 
\bigl(\rho_\veps^2\,u_\veps\bigr)_\veps\;\subset\;L^2_T\bigl(L^2(\Omega)\bigr)\,.
$$ 

\subsection{Study of the singular perturbation operator}

In the present subsection we identify some properties and constraints on the limit points of the family of weak solutions
$\bigl(\rho_\veps,u_\veps\bigr)_\veps$.

\medbreak
First of all, by uniform bounds we immediately deduce that $\rho_\veps\,\ra\,1$ (strong convergence)
in $L^\infty\bigl(\R_+;H^1(\Omega)\bigr)\,\cap\,L^2_{loc}\bigl(\R_+;H^2(\Omega)\bigr)$, with convergence rate of order $\veps$.
So, we can write $\rho_\veps\,=\,1\,+\,\veps\,r_\veps$, with the family $\bigl(r_\veps\bigr)_\veps$ bounded in
the previous spaces. Then we infer that
\begin{equation} \label{eq:conv-r}
\qquad\qquad\qquad
r_\veps\,\rightharpoonup\,r\qquad\qquad\qquad\mbox{ in }\quad
L^\infty\bigl(\R_+;H^1(\Omega)\bigr)\,\cap\,L^2_{loc}\bigl(\R_+;H^2(\Omega)\bigr)\,.
\end{equation}

In the same way, if we define $a_\veps\,:=\,\bigl(1/\rho_\veps-1\bigr)/\veps$, we gather that $\bigl(a_\veps\bigr)_\veps$ is uniformly
bounded in $L^\infty\bigl(\R_+;L^2(\Omega)\bigr)\,\cap\,L^2_{loc}\bigl(\R_+;H^1(\Omega)\bigr)$. So it weakly converges
to some $a$ in this space: arguing as done in Remark \ref{r:cold-density}, it is easy to check that
\begin{equation} \label{eq:conv-a}
\qquad\qquad\qquad
a_\veps\,\rightharpoonup\,-\,r\qquad\qquad\qquad\mbox{ in }\quad
L^\infty\bigl(\R_+;L^2(\Omega)\bigr)\,\cap\,L^2_{loc}\bigl(\R_+;H^1(\Omega)\bigr)\,.
\end{equation}

Again by uniform bounds, we also deduce
\begin{equation}
\qquad\qquad\qquad
u_\veps\,\rightharpoonup\,u\,\qquad\qquad\qquad\mbox{ in }\quad L^2_{loc}\bigl(\R_+;L^{3/2}_{loc}(\Omega)\bigr) \label{eq:conv-u}
\end{equation}
and $Du_\veps\,\rightharpoonup\,Du$ in $L^2_{loc}\bigl(\R_+;L^1_{loc}(\Omega)\bigr)$, where we have identified $(L^1)^*$ with
$L^\infty$.

Notice also that, by uniqueness of the limit, we have the additional properties
\begin{eqnarray*}
\sqrt{\rho_\veps}\,u_\veps\,\stackrel{*}{\rightharpoonup}\,u & \qquad\qquad\mbox{ in }\quad &  L^\infty\bigl(\R_+;L^2(\Omega)\bigr) \\[1ex]
\rho_\veps\,u_\veps\,\rightharpoonup\,u & \qquad\qquad\mbox{ in }\quad &  L^2_{loc}\bigl(\R_+;L^2(\Omega)\bigr) \\[1ex]
\sqrt{\rho_\veps}\,Du_\veps\,\rightharpoonup\,Du & \qquad\qquad\mbox{ in }\quad &  L^2\bigl(\R_+;L^2(\Omega)\bigr)\,,
\end{eqnarray*}
where $\stackrel{*}{\rightharpoonup}$ denotes the weak-$*$ convergence in $L^\infty\bigl(\R_+;L^2(\Omega)\bigr)$.

\medbreak
Let us find now some constraints the limit points $(r,u)$ have to satisfy: the following result can be seen as
the analogue of the Taylor-Proudman theorem in our setting.

\begin{prop} \label{p:TP}
Let $\bigl(\rho_\veps,u_\veps\bigr)_\veps$ be a family of weak solutions to system \eqref{eq:NSK-sing}-\eqref{eq:bc}, each one
related to the initial datum $\bigl(\rho_{0,\veps},u_{0,\veps}\bigr)$ and fulfilling the hypotheses fixed in Section \ref{s:results}. \\
Let us define $r_\veps:=\veps^{-1}\left(\rho_\veps-1\right)$, and let $(r,u)$ be a limit point of the sequence
$\bigl(r_\veps,u_\veps\bigr)_\veps$.

Then $r=r(x^h)$ and $u=\bigl(u^h(x^h),0\bigr)$, with $\div_{\!h}u^h=0$. Moreover, we infer the properties
$$
\mf{c}\,u^h\;=\,\nabla_h^\perp\bigl(\Id\,-\,\Delta_h\bigr)r\qquad\qquad\mbox{ and }\qquad\qquad
u^h\,\cdot\,\nabla_h\mf{c}\,\equiv\,0\,.
$$
\end{prop}

\begin{proof}
First of all, we test the mass equation on any $\phi\in\mc{D}\bigl([0,T[\,\times\Omega\bigr)$: using
the decomposition $\rho_\veps\,=\,1\,+\,\veps\,r_\veps$, we get
$$
-\,\veps\int^T_0\int_\Omega r_\veps\,\d_t\phi\,-\,\int^T_0\int_\Omega\rho_\veps\,u_\veps\cdot\nabla\phi\,=\,
\veps\int_\Omega r_{0,\veps}\,\phi(0)\,.
$$
From this relation, letting $\veps\ra0$, by uniform bounds and convergence properties established above, we deduce that
$\int^T_0\int_\Omega u\cdot\nabla\phi\,=\,0$, which in turn implies
\begin{equation} \label{constr:div}
\div\,u\;\equiv\:0\qquad\qquad\mbox{ almost everywhere in }\qquad [0,T]\times\Omega\,.
\end{equation}

Let us consider now the velocity field: for any $\psi\in\mc{D}\bigl([0,T[\,\times\Omega;\R^3\bigr)$, we use $\veps\,\psi$
as a test-function in the momentum equation \eqref{eq:weak-momentum}. By use of uniform bounds, we see that the
only integrals which do not go to $0$ are the ones involving the pressure $\Pi$, the Coriolis force $\mf{C}$ and the capillarity term:
let us focus on them.

We start by dealing with the classical part $P$ of the pressure: we can write
$$
\frac{1}{\veps}\int^T_0\!\!\!\int_\Omega\nabla P(\rho_\veps)\cdot\psi\,=\,-\frac{1}{\veps}\int^T_0\!\!\!\int_\Omega
\biggl(P(\rho_\veps)-P(1)-P'(1)\left(\rho_\veps-1\right)\biggr)\div\psi+
\frac{1}{\veps}\,P'(1)\int^T_0\!\!\!\int_\Omega\nabla\rho_\veps\cdot\psi\,.
$$
The quantity $P(\rho_\veps)-P(1)-P'(1)\left(\rho_\veps-1\right)$ can be controlled by the internal
energy $h(\rho_\veps)$: since $h(\rho_\veps)/\veps^2$ is uniformly bounded in $L^\infty_T\bigl(L^1\bigr)$
(see Propositions \ref{p:E} and \ref{p:sing_b-E}), the first integral tends to $0$ for $\veps\ra0$. Hence, thanks also to
\eqref{eq:conv-r} we find
$$
\frac{1}{\veps}\int^T_0\int_\Omega\nabla P(\rho_\veps)\cdot\psi\;\longrightarrow\;
P'(1)\int^T_0\int_\Omega\nabla r\cdot\psi\,.
$$
An analogous decomposition allows us to treat also the cold part of the pressure: we find that
$(1/\veps)\int^T_0\int_\Omega\nabla P_c(\rho_\veps)\cdot\psi$ converges to $P_c'(1)\int^T_0\int_\Omega\nabla r\cdot\psi$. Therefore,
since $P'(1)=P'_c(1)=1/2$, in the end we infer the convergence, in the limit for $\veps$ going to $0$,
$$
\frac{1}{\veps}\int^T_0\int_\Omega\nabla\Pi(\rho_\veps)\cdot\psi\;\longrightarrow\;
\int^T_0\int_\Omega\nabla r\cdot\psi\,.
$$

As for the rotation term, uniform bounds immediately imply
$$
\int^T_0\int_\Omega\mf{C}(\rho_\veps,u_\veps)\cdot\psi\,=\,
\int^T_0\int_\Omega\mf{c}\,e^3\times\rho_\veps\,u_\veps\cdot\psi\;\longrightarrow\;
\int^T_0\int_\Omega\mf{c}\,e^3\times u\cdot\psi\,.
$$

Finally, let us consider the integrals coming from the capillarity tensor. The former one can be treated writing, as usual,
$\rho_\veps\,=\,1\,+\,(\rho_\veps-1)$: therefore  we get
$$
\frac{1}{\veps}\int^T_0\int_\Omega\rho_\veps\,\Delta\rho_\veps\,\cdot\,\div\psi
\;\longrightarrow\;\int^T_0\int_\Omega\Delta r\,\cdot\,\div\psi\,,
$$
where we used the uniform $L^2_T\bigl(L^2\bigr)$ bound on $\Delta\rho_\veps$ and the fact that $\rho_\veps-1$ is
of order $\veps$ in e.g. $L^\infty_T\bigl(L^2\bigr)$.
On the other hand, since also $\nabla\rho_\veps$ is of order $\veps$ in $L^2_T\bigl(L^2\bigr)$, one easily gets
$$
\frac{1}{\veps}\int^T_0\int_\Omega\Delta\rho_\veps\,\nabla\rho_\veps\,\cdot\,\psi\;\longrightarrow\;0\,.
$$

Let us sum up all these informations: from the momentum equation, when $\veps\ra0$, we have found that the limit point
$(r,u)$ has to verify the relation
\begin{equation} \label{constr:u-r}
\mf{c}(x^h)\,e^3\,\times\,u\,+\,\nabla\bigl(\Id\,-\,\Delta\bigr)r\,=\,0\,,
\end{equation}
which means in particular
$$
\begin{cases}
 \d_1\bigl(\Id\,-\,\Delta\bigr)r\,=\,\mf{c}\,u^2 \\
 \d_2\bigl(\Id\,-\,\Delta\bigr)r\,=\,-\,\mf{c}\,u^1 \\
 \d_3\bigl(\Id\,-\,\Delta\bigr)r\,=\,0\,.
\end{cases}
$$
This relation immediately implies that $\bigl(\Id\,-\,\Delta\bigr)r\,=\,\bigl(\Id\,-\,\Delta\bigr)r(x^h)$ depends just on the horizontal
variables. Hence, $\d_3r$ fulfills the elliptic equation $-\Delta\d_3r+\d_3r\,=\,0$  in $\Omega$:
by passing in Fourier variables on $\R^2\times\mbb{T}^1$, we deduce that
\begin{equation} \label{constr:r}
\d_3r\,\equiv\,0\qquad\qquad\Longrightarrow\qquad\qquad r\,=\,r(x^h)\,.
\end{equation}

On the other hand, by \eqref{constr:u-r} we get that also $\mf{c}\,u^h$ has to depend just on the horizontal variables;
taking the $\d_3$ derivative (since $\mf{c}\,=\,\mf{c}(x^h)$ and $\mf{c}\neq0$ almost everywhere on $\R^2$) implies that
$u^h\,=\,u^h(x^h)$.

From this property and the divergence-free condition \eqref{constr:div}, it follows also that $\d_3^2u^3\,\equiv\,0$, i.e.
$\d_3u^3\,=\,\d_3u^3(x^h)$. On the other hand, by periodicity $\int_{\mbb{T}^1}\d_3u^3\,dx^3\,=\,0$, which entails $\d_3u^3\,=\,0$
by decomposition \eqref{dec:vert-av}, and then $u^3\,=\,u^3(x^h)$. By use of the complete-slip
boundary conditions \eqref{eq:bc}, in the end we deduce that $u^3\,\equiv\,0$, which in turn gives
\begin{equation} \label{constr:u_1}
 u\,=\,\bigl(u^h(x^h),0\bigr)\,,\qquad\qquad\mbox{ with }\qquad \div_{\!h}u^h\,=\,0\,.
\end{equation}

Finally, let us apply the ${\rm rot}$ operator to equation \eqref{constr:u-r}: we obtain
\begin{equation} \label{constr:u_2}
\div_{\!h}\bigl(\mf{c}(x^h)\,u^h\bigr)\,\equiv\,0\qquad\qquad\Longrightarrow\qquad\qquad
u^h\cdot\nabla_h\mf{c}\,=\,0\,,
\end{equation}
where us have used the just proved property \eqref{constr:u_1}.
\end{proof}

\begin{rem} \label{r:limit-dens}
Notice that, by the previous proposition and the fact that $r,u\,\in\,L^\infty\bigl(\R_+;L^2(\Omega)\bigr)$, we get that
actually $r\in L^\infty\bigl(\R_+;H^3(\Omega)\bigr)$.
\end{rem}

\subsection{Passing to the limit in the weak formulation} \label{ss:variable}

By Proposition \ref{p:TP}, we can define the singular perturbation operator
$$
\begin{array}{lccc}
\wtilde{\mc{A}}\,: & L^2(\Omega)\;\times\;L^2(\Omega) & \longrightarrow & H^{-1}(\Omega)\;\times\;H^{-3}(\Omega) \\
& \bigl(\,r\;,\;V\,\bigr) & \mapsto & \Bigl(\div V\;,\;\mf{c}(x^h)\,e^3\times V\,+\nabla\bigl(\Id-\Delta\bigr)r\Bigr)\,,
\end{array}
$$
which has variable coefficients: so, spectral analysis tools (employed in \cite{F}) are out of use here. Then,
in order to prove convergence in the weak formulation of our equations, we will resort then to a compensated compactness argument.

This technique goes back to works by P.-L. Lions and Masmoudi about incompressible limit (see e.g. \cite{L-M}-\cite{L-M_CRAS}); it was introduced for
the first time in the context of highly rotating fluids by Gallagher and Saint-Raymond in \cite{G-SR_2006}. There, they dealt with incompressible Navier-Stokes equations
with variable rotation axis. With the same strategy, in paper \cite{F-G-GV-N} Feireisl, Gallagher, G\'erard-Varet and Novotn\'y studied the case of
a compressible Navier-Stokes system with Earth rotation and centrifugal force, when the fixed limit density profile is supposed non-constant.

Most of our analysis will follow the one performed in \cite{G-SR_2006}-\cite{F-G-GV-N}. Notice however that here we have to deal with an
additional term, coming from the presence of capillarity.

\medbreak
Let us consider tests functions $\phi\,\in\,\mc{D}\bigl([0,T[\,\times\Omega\bigr)$ and
$\psi\,\in\,\mc{D}\bigl([0,T[\,\times\Omega;\R^3\bigr)$ such that the couple $(\phi,\psi)$ belongs to
${\rm Ker}\,\wtilde{\mc A}$. Recall that, by Proposition \ref{p:TP}, they satisfy
$$
\div\psi\,=\,0\qquad\mbox{ and }\qquad \mf{c}(x^h)\,e^3\times\psi\,+\,\nabla\bigl(\Id-\Delta\bigr)\phi\,=\,0\,.
$$
In particular, $\psi=\bigl(\psi^h,0\bigr)$ and $\phi$ just depend on the horizontal variable $x^h\in\R^2$ and they are linked
by the relation $\mf{c}\,\psi^h\,=\,\nabla^\perp_h\bigl(\Id-\Delta_h\bigr)\phi$. Finally, combining this property
with the divergence-free condition for $\psi$, we infer also that
$\nabla^\perp_h\bigl(\Id-\Delta_h\bigr)\phi\cdot\nabla_h\mf{c}\,=\,0$.

First of all, we evaluate the momentum equation on such a $\psi$: taking into account the previous properties, we end up with
\begin{eqnarray}
\int_\Omega\rho_{0,\veps}\,u_{0,\veps}\cdot\psi(0)\,dx & = & 
\int^T_0\!\!\int_\Omega\biggl(-\rho_\veps\,u_\veps\cdot\d_t\psi\,-\,\rho_\veps\,u_\veps\otimes u_\veps:\nabla\psi\,+ \label{eq:sing_weak} \\
& & 
+\,\nu\,\rho_\veps\,Du_\veps:\nabla\psi\,+\,\frac{1}{\veps^2}\,\Delta\rho_\veps\,\nabla\rho_\veps\cdot\psi\,+\,
\frac{\mf{c}(x^h)}{\veps}\,e^3\times\rho_\veps\,u_\veps\cdot\psi\biggr)dx\,dt. \nonumber
\end{eqnarray}
Notice that the $\d_t$ and viscosity terms do not present any difficulty in passing to the limit. On the other hand, the rotation
term can be handled by use of the weak form of the mass equation, tested on
$\wtilde{\phi}\,=\,\bigl(\Id-\Delta_h\bigr)\phi$: we get
\begin{eqnarray*}
\frac{1}{\veps}\int^T_0\!\!\int_\Omega\mf{c}(x^h)\,e^3\times\rho_\veps\,u_\veps\cdot\psi & = & 
-\,\frac{1}{\veps}\int^T_0\!\!\int_\Omega\mf{c}(x^h)\,\rho_\veps\,u^h_\veps\cdot\left(\psi^h\right)^\perp \\
& = & \frac{1}{\veps}\int^T_0\!\!\int_\Omega\rho_\veps\,u^h_\veps\cdot\nabla_h\wtilde{\phi}\;=\;
-\int_\Omega r_{0,\veps}\,\wtilde{\phi}(0)\,-\,\int^T_0\!\!\int_\Omega r_\veps\,\d_t\wtilde{\phi}\,,
\end{eqnarray*}
which obviously converges in the limit $\veps\ra0$.

In order to deal with the transport and the capillarity terms, we want to use the structure of the system. Therefore, first of all we need
to introduce a regularization of our solutions.

\subsubsection{Regularization and description of the oscillations}

Let us set $V_\veps\,:=\,\rho_\veps\,u_\veps$. We can write system \eqref{eq:NSK-sing} in the form
\begin{equation} \label{eq:ac-w_c}
\begin{cases}
\veps\,\d_tr_\veps\,+\,\div\,V_\veps\,=\,0 \\[1ex]
\veps\,\d_tV_\veps\,+\,\Bigl(\mf{c}(x^h)\,e^3\times V_\veps\,+\,\nabla\bigl(\Id\,-\,\Delta\bigr)r_\veps\Bigr)\,=\,\veps\,f_\veps\,,
\end{cases}
\end{equation}
where we have defined $f_\veps$ by the formula
\begin{eqnarray}
f_\veps & := & -\,\div\left(\rho_\veps u_\veps\otimes u_\veps\right)\,+\,\nu\,\div\left(\rho_\veps Du_\veps\right)\,- 
\label{eq:f_veps} \\
& & \qquad\qquad -\,\frac{1}{\veps^2}\nabla\Bigl(\Pi(\rho_\veps)-\Pi(1)-\Pi'(1)\left(\rho_\veps-1\right)\Bigr)\,+\,
\frac{1}{\veps^2}\bigl(\rho_\veps-1\bigr)\,\nabla\Delta\rho_\veps\,. \nonumber
\end{eqnarray}
Equations \eqref{eq:ac-w_c} have to be read in the weak sense, of course. In particular, from writing
\begin{eqnarray*}
\langle f_\veps,\psi\rangle & := & \int_\Omega\biggl(\rho_\veps u_\veps\otimes u_\veps:\nabla\psi\,-\,
\nu\,\rho_\veps Du_\veps:\nabla\psi\,-\,\frac{1}{\veps^2}\,\Delta\rho_\veps\,\nabla\rho_\veps\cdot\psi\,- \\
& & \qquad\qquad -\,\frac{1}{\veps^2}\,(\rho_\veps-1)\,\Delta\rho_\veps\,\div\psi\,+\,
\frac{1}{\veps^2}\Bigl(\Pi(\rho_\veps)-\Pi(1)-\Pi'(1)\left(\rho_\veps-1\right)\Bigr)\div\psi\biggr)dx \\
& = & \int_\Omega\Bigl(f^1_\veps:\nabla\psi\,+\,f^2_\veps:\nabla\psi\,+\,f^3_\veps\cdot\psi\,+\,f^4_\veps\,\div\psi\,+\,
f^5_\veps\,\div\psi\Bigr)\,dx
\end{eqnarray*}
and by a systematic use of uniform bounds (recall Propositions \ref{p:sing_b-E} and \ref{p:sing_b-BD}),
we can easily see that $\bigl(f^1_\veps\bigr)_\veps$ and $\bigl(f^5_\veps\bigr)_\veps$ are uniformly bounded in $L^\infty_T\bigl(L^1\bigr)$,
and so is $\bigl(f^2_\veps\bigr)_\veps$ in $L^2_T\bigl(L^2\bigr)$; finally, $\bigl(f^3_\veps\bigr)_\veps$ and $\bigl(f^4_\veps\bigr)_\veps$
are bounded in $L^2_T\bigl(L^1\bigr)$.

Therefore, we deduce the uniform boundedness of $\bigl(f_\veps\bigr)_\veps$ in the space
$L^2_T\bigl(H^{-1}(\Omega)\,+\,W^{-1,1}(\Omega)\bigr)$, and then in particular in $L^2_T\bigl(H^{-s}(\Omega)\bigr)$ for any
$s>5/2$.

\medbreak
Now, for any $M>0$, let us consider the low-frequency cut-off operator $S_M$ of a Littlewood-Paley decomposition,
as introduced in \eqref{eq:low-freq} below, and let us define
$$
r_{\veps,M}\,:=\,S_Mr_\veps\qquad\qquad\mbox{ and }\qquad\qquad
V_{\veps,M}\,:=\,S_MV_\veps\,.
$$
The following result hods true.
\begin{prop} \label{p:regular}
For any fixed time $T>0$ and compact set $K\subset\Omega$, the  following convergence properties hold, in the limit for $M\longrightarrow+\infty$:
\begin{equation} \label{reg:convergence}
\begin{cases}
\; \sup_{\veps>0}\left\|r_\veps\,-\,r_{\veps,M}\right\|_{L^\infty_T(H^s(K))\,\cap\,L^2_T(H^{1+s}(K))}\,\longrightarrow\,0
\qquad\qquad\forall\; s<1 \\[1ex]
\; \sup_{\veps>0}\left\|V_\veps\,-\,V_{\veps,M}\right\|_{L^2_T(H^{-s}(K))}\,\longrightarrow\,0
\qquad\qquad\forall\; s>0\,.
\end{cases}
\end{equation}

Moreover, for any $M>0$, the couple $\bigl(r_{\veps,M}\,,\,V_{\veps,M}\bigr)$ satisfies the approximate wave equations
\begin{equation} \label{reg:approx-w}
\begin{cases}
\veps\,\d_tr_{\veps,M}\,+\,\div\,V_{\veps,M}\,=\,0 \\[1ex]
\veps\,\d_tV_{\veps,M}\,+\,\Bigl(\mf{c}(x^h)\,e^3\times V_{\veps,M}\,+\,\nabla\bigl(\Id-\Delta\bigr)r_{\veps,M}\Bigr)\,=\,
\veps\,f_{\veps,M}\,+\,g_{\veps,M}\,,
\end{cases}
\end{equation}
where $\bigl(f_{\veps,M}\bigr)_{\veps}$ and $\bigl(g_{\veps,M}\bigr)_{\veps}$ are families of smooth functions satisfying
\begin{equation} \label{reg:source}
\begin{cases}
\; \sup_{\veps>0}\left\|f_{\veps,M}\right\|_{L^2_T(H^{s}(K))}\,\leq\,C(s,M)
\qquad\qquad\forall\; s\geq0 \\[1ex]
\; \sup_{\veps>0}\left\|g_{\veps,M}\right\|_{L^2_T(H^1(K))}\,\longrightarrow\,0
\qquad\qquad\mbox{ for }\quad M\ra+\infty\,,
\end{cases}
\end{equation}
where the constant $C(s,M)$ depends on the fixed values of $s\geq0$, $M>0$.
\end{prop}

\begin{proof}
Keeping in mind the characterization of $H^s$ spaces in terms of Littlewood-Paley decomposition (see Appendix \ref{app:LP}),
properties \eqref{reg:convergence} are straightforward consequences of the uniform bounds established in Subsection \ref{ss:uniform}.

Next, applying operator $S_M$ to \eqref{eq:ac-w_c} immediately gives us system \eqref{reg:approx-w},
where, denoting by $[\mc{P},\mc{Q}]$ the commutator between two operators $\mc P$ and $\mc Q$, we have set
$$
f_{\veps,M}\,:=\,S_Mf_\veps\qquad\qquad\mbox{ and }\qquad\qquad 
g_{\veps,M}\,:=\,\bigl[\mf{c}(x^h),S_M\bigr]\bigl(e^3\times V_\veps\bigr)\,.
$$
By these definitions and the uniform bounds on $\bigl(f_\veps\bigr)_\veps$, it is easy to verify the first property in \eqref{reg:source}.
As for the second one, we need to generalize the argument of Lemma 3.3 in \cite{Masm}, because here we have less available
controls on the family $\bigl(V_\veps\bigr)_\veps$.

First of all, by uniform bounds and Lemma \ref{l:commutator} we get
$$
\sup_{\veps>0}\left\|g_{\veps,M}\right\|_{L^2_T(L^2)}\,\leq\,C\,2^{-M}\,.
$$
As for the gradient, for any $1\leq j\leq 3$ we can write
$$
\d_jg_{\veps,M}\,=\,\left[\mf{c}\,,\,S_M\right]\d_j\left(e^3\times V_\veps\right)\,+\,
\left[\d_j\mf{c}\,,\,S_M\right]\left(e^3\times V_\veps\right)\,.
$$
In order to control the former term, we use Lemma \ref{l:comm_L^p} with $p_2=q=2$ and $p_1=1$. Recalling that, by \eqref{sing-b:D_rho-u},
$\left(DV_\veps\right)_\veps\subset L^2_T(L^1_{loc})$, for any compact $K\subset\Omega$ we get
$$
\sup_{\veps>0}\left\|\left[\mf{c}\,,\,S_M\right]\d_j\left(e^3\times V_\veps\right)\right\|_{L^2_T(L^2(K))}\,\leq\,C\,2^{-M}\,.
$$
For the latter term, instead, Lemma \ref{l:comm-less} gives us
$$
\sup_{\veps>0}\left\|\left[\d_j\mf{c}\,,\,S_M\right]\left(e^3\times V_\veps\right)\right\|_{L^2_T(L^2)}\,\leq\,C\,\mu(2^{-M})\,.
$$

In the end, choosing $\eta(M)=\max\left\{2^{-M},\mu(2^{-M})\right\}$ (which goes to $0$ when $M\ra+\infty$), we get
$$
\sup_{\veps>0}\left\|g_{\veps,M}\right\|_{L^2_T(H^1_{loc})}\,\leq\,C\,\eta(M)
$$
for a suitable constant $C>0$, and this completes the proof of the proposition.
\end{proof}

We also have an important decomposition for the approximated velocity fields.
\begin{prop} \label{p:decomp}
The following decompositions hold true: 
$$
V_{\veps,M}\,=\,\mbb{V}_{\veps,M}\,+\,\veps\,\mc{V}_{\veps,M}\qquad\qquad\mbox{ and }\qquad\qquad
DV_{\veps,M}\,=\,\mbb{D}_{\veps,M}\,+\,\veps\,\mc{D}_{\veps,M}\,,
$$
where, for any compact set $K\subset\Omega$ and any $s\geq0$ one has
$$
\begin{cases}
\left\|\mbb{V}_{\veps,M}\right\|_{L^2_T\bigl(L^2(K)\cap L^3(K)\bigr)}\,+\,
\left\|\mbb{D}_{\veps,M}\right\|_{L^2_T\bigl(L^2(K)\bigr)}\,\leq\,C(K) \\[1ex]
\left\|\mc{V}_{\veps,M}\right\|_{L^2_T\bigl(H^s(K)\bigr)}\,+\,
\left\|\mc{D}_{\veps,M}\right\|_{L^2_T\bigl(H^s(K)\bigr)}\,\leq\,C(K,s,M)\,,
\end{cases}
$$
for suitable positive constants $C(K)$, $C(K,s,M)$ depending just on the quantities in the brackets.
\end{prop}

\begin{proof}
By definitions, we immediately have
$$
V_{\veps,M}\,=\,S_M\bigl(\rho_\veps\,u_\veps\bigr)\,=\,S_M\bigl(\rho_\veps^{3/2}\,u_\veps\bigr)\,-\,
S_M\bigl((\sqrt{\rho_\veps}-1)\,\rho_\veps\,u_\veps\bigr)\,.
$$
Thanks to the uniform bounds established in Subsection \ref{ss:uniform}, the first decomposition and the related estimates are easy
to be verified.

Let us now take a space derivative of $V_{\veps,M}$, splitted in accordance with the previous identity. Thanks to spectral localization,
the second term do not present any problem: indeed, for any $1\leq j\leq3$ one has
\begin{eqnarray*}
\veps^{-1}\,\d_j\bigl((\sqrt{\rho_\veps}-1)\,\rho_\veps\,u_\veps\bigr) & = & \d_j\frac{\sqrt{\rho_\veps}-1}{\veps}\;\rho_\veps\,u_\veps\,+\,
\frac{\sqrt{\rho_\veps}-1}{\veps}\,\d_j\bigl(\rho_\veps\,u_\veps\bigr) \\
& = & \frac{\d_j\rho_\veps}{2\,\veps}\,\sqrt{\rho_\veps}\,u_\veps\,+\,\frac{\sqrt{\rho_\veps}-1}{\veps}\,\d_j\bigl(\rho_\veps\,u_\veps\bigr)\,.
\end{eqnarray*}
For the former term, one uses Proposition \ref{p:sing_b-E} and the fact that (by Sobolev embeddings) $\veps^{-1}\nabla\rho_\veps\,\in\,L^2_T(L^6)$; for the latter,
instead, one has to take advantage of the estimate $|\sqrt{\rho}-1|\leq|\rho-1|$ together with property \eqref{sing-b:D_rho-u}.

As for $S_M\bigl(\rho_\veps^{3/2}\,u_\veps\bigr)$, instead, we have to proceed carefully. More precisely,
for any $1\leq  j\leq 3$ we start by writing
\begin{eqnarray*}
\d_j S_M\bigl(\rho_\veps^{3/2}\,u_\veps\bigr) & = & \frac{3}{2}\,S_M\bigl(\sqrt{\rho_\veps}\,u_\veps\,\d_j\rho_\veps\bigr)\,+\,
S_M\bigl(\rho_\veps^{3/2}\,\d_ju_\veps\bigr) \\[1ex]
& = & \frac{3}{2}\,S_M\bigl(\sqrt{\rho_\veps}\,u_\veps\,\d_j\rho_\veps\bigr)\,+\,S_M\bigl(\sqrt{\rho_\veps}\,\d_ju_\veps\bigr)\,+\,
S_M\bigl((\rho_\veps-1)\,\sqrt{\rho_\veps}\,\d_ju_\veps\bigr)\,.
\end{eqnarray*}
Recalling now that
$\bigl(\sqrt{\rho_\veps}\,Du_\veps\bigr)_\veps\,\subset\,L^2_T\bigl(L^2(\Omega)\bigr)$, we can set $\mbb{D}_{\veps,M}\,=\,S_M\bigl(\sqrt{\rho_\veps}\,\d_ju_\veps\bigr)$.
Indeed, since $\bigl(\veps^{-1}\,(\rho_\veps-1)\bigr)_\veps\,\subset\,L^\infty_T\bigl(L^6(\Omega)\bigr)$ and 
$\bigl(\veps^{-1}\,\nabla\rho_\veps\bigr)_\veps\,\subset\,L^\infty_T\bigl(L^2(\Omega)\bigr)$,
the other two terms are of order $\veps$. So we can include them into the remainder $\mc{D}_{\veps,M}$.

Hence, the proposition is now proved.
\end{proof}

\subsubsection{The capillarity term}

First of all, let us deal with the surface tension term in \eqref{eq:sing_weak}. Notice that it can be rewritten as
$\;\int^T_0\int_\Omega\Delta r_\veps\,\nabla r_\veps\cdot\psi\;$, for any smooth test function $\psi$.

Thanks to the next lemma, we reconduct ourselves to study the convergence in the case of regular density functions.
\begin{lemma} \label{l:capill_approx}
For any $\psi\,\in\,\mc{D}\bigl([0,T[\,\times\Omega;\R^3\bigr)$, we have
$$
\lim_{M\ra+\infty}\,\limsup_{\veps\ra0}\,\left|\int^T_0\!\!\int_\Omega\Delta r_\veps\;\nabla r_\veps\,\cdot\,\psi\,dx\,dt\,-\,
\int^T_0\!\!\int_\Omega \Delta r_{\veps,M}\; \nabla r_{\veps,M}\,\cdot\,\psi\,dx\,dt\right|\,=\,0\,.
$$
\end{lemma}

\begin{proof}
Let us fix $M>0$: we can write
$$
\int^T_0\!\!\int_\Omega\Delta r_\veps\;\nabla r_\veps\cdot\psi\,=\,\int^T_0\!\!\int_\Omega\Delta r_\veps\;\nabla r_{\veps,M}\cdot\psi\,+\,
\int^T_0\!\!\int_\Omega\Delta r_\veps\;\nabla\bigl(\Id-S_M\bigr)r_\veps\cdot\psi\,.
$$
By the uniform $L^2_T\bigl(L^2\bigr)$ bounds on the family $\bigl(\Delta r_\veps\bigr)_\veps$ and the first property in
\eqref{reg:convergence}, we get that
$$
\left|\int^T_0\!\!\int_\Omega\Delta r_\veps\;\nabla\bigl(\Id-S_M\bigr)r_\veps\cdot\psi\right|\,\leq\,C\,\delta(M)\,,
$$
for some positive constant $C$, independent of $\veps$ and $M$, and some function $\delta(M)$ such that
$\delta(M)\longrightarrow0$ for $M\ra+\infty$.

On the other hand, the former term in the right-hand side of the previous identity can be written as
$$
\int^T_0\!\!\int_\Omega\Delta r_\veps\;\nabla r_{\veps,M}\cdot\psi\,=\,
\int^T_0\!\!\int_\Omega\Delta r_{\veps,M}\;\nabla r_{\veps,M}\cdot\psi\,+\,
\int^T_0\!\!\int_\Omega\Delta\bigl(\Id-S_M\bigr)r_\veps\;\nabla r_{\veps,M}\cdot\psi\,.
$$
For the last term, using that the operator $\Id-S_M$ is self-adjoint, we can estimate
$$
\left|\int^T_0\!\!\int_\Omega\Delta\bigl(\Id-S_M\bigr)r_\veps\;\nabla r_{\veps,M}\cdot\psi\right|\,\leq\,
C\,\left\|\Delta r_\veps\right\|_{L^2_T(L^2)}\,\left\|\bigl(\Id-S_M\bigr)\bigl(\nabla r_{\veps,M}\cdot\psi\bigr)\right\|_{L^2_T(L^2)}\,,
$$
and this quantity converges to $0$ for $M\longrightarrow+\infty$, thanks to Lemma \ref{l:density}, point $(iii)$. Indeed, it is enough
to notice that $\left(\nabla\bigl(\nabla r_{\veps,M}\cdot\psi\bigr)\right)_{\veps,M}$ is bounded
in $L^2_T(L^2)$, uniformly both in $\veps$ and $M$.
\end{proof}

Thanks to Lemma \ref{l:capill_approx}, for any $\psi\,\in\,\mc{D}\bigl([0,T[\,\times\Omega;\R^3\bigr)\,\cap\,{\rm Ker}\,\wtilde{\mc{A}}$
we have to consider the convergence of the term (pay attention to the signs)
\begin{eqnarray*}
\hspace{-0.5cm} \int^T_0\!\!\int_\Omega\Delta r_{\veps,M}\;\nabla r_{\veps,M}\cdot\psi\,dx\,dt & = &
-\,\int^T_0\!\!\int_\Omega\bigl(\Id-\Delta\bigr)r_{\veps,M}\;\nabla r_{\veps,M}\cdot\psi\,dx\,dt\,+ \\
& & \qquad\qquad\qquad\qquad +\,\int^T_0\!\!\int_\Omega r_{\veps,M}\;\nabla r_{\veps,M}\cdot\psi\,dx\,dt\,.
\end{eqnarray*}
Notice that $r_{\veps,M}\,\nabla r_{\veps,M}\,=\,\nabla\left(r_{\veps,M}\right)^2/2$: therefore, since $\psi$ is divergence-free,
by integration by parts we get that the latter item on the right-hand side is identically $0$.

Therefore, in the end we have to deal only with the remainder
\begin{eqnarray}
-\,\int^T_0\!\!\int_\Omega\bigl(\Id-\Delta\bigr)r_{\veps,M}\;\nabla r_{\veps,M}\cdot\psi & = & 
-\,\int^T_0\!\!\int_\Omega\bigl(\Id-\Delta_h\bigr)\lan r_{\veps,M}\ran\;\nabla_h\lan r_{\veps,M}\ran\cdot\psi\,- \label{comp-cpt:dens} \\
& & \qquad\qquad
-\,\int^T_0\!\!\int_\Omega\lan\bigl(\Id-\Delta\bigr)\wtilde{r}_{\veps,M}\;\nabla\wtilde{r}_{\veps,M}\ran\cdot\psi\,, \nonumber
\end{eqnarray}
where we used the notations introduced in \eqref{dec:vert-av}, setting $\wtilde{r}_{\veps,M}$ the mean-free part of $r_{\veps,M}$.

\subsubsection{The convective term} \label{sss:convect}

In the present paragraph we will deal with the convective term. Once again, the first step is to reduce
the study to the case of smooth vector fields $V_{\veps,M}$.

\begin{lemma} \label{l:conv_approx}
For any $\psi\,\in\,\mc{D}\bigl([0,T[\,\times\Omega;\R^3\bigr)$, we have
$$
\lim_{M\ra+\infty}\,\limsup_{\veps\ra0}\,\left|\int^T_0\!\!\int_\Omega\rho_\veps u_\veps\otimes u_\veps:\nabla\psi\,dx\,dt\,-\,
\int^T_0\!\!\int_\Omega V_{\veps,M}\otimes V_{\veps,M}:\nabla\psi\,dx\,dt\right|\,=\,0\,.
$$
\end{lemma}

\begin{proof}
First of all, we can write
$$
\int^T_0\!\!\int_\Omega\rho_\veps u_\veps\otimes u_\veps:\nabla\psi\,=\,
\int^T_0\!\!\int_\Omega\rho_\veps u_\veps\otimes\sqrt{\rho_\veps}\,u_\veps:\nabla\psi\,-\,
\veps\int^T_0\!\!\int_\Omega\frac{\sqrt{\rho_\veps}-1}{\veps}\,\rho_\veps u_\veps\otimes u_\veps:\nabla\psi
$$
Notice that the latter integral in the right-hand side is of order $\veps$: this is a consequence of the uniform bounds
$\bigl(\sqrt{\rho_\veps}\,u_\veps\bigr)_\veps\,\subset\,L^\infty_T\bigl(L^2\bigr)$ and
$\bigl(\veps^{-1}(\sqrt{\rho_\veps}-1)\bigr)_\veps\,\subset\,L^p_T\bigl(L^\infty\bigr)$ for any $p\in[2,4[\,$ (recall Proposition
\ref{p:sing_b-BD}). The former one, instead, can be decomposed again into
\begin{eqnarray*}
& & \hspace{-0.5cm}
\int^T_0\!\!\int_\Omega\rho_\veps u_\veps\otimes\sqrt{\rho_\veps}\,u_\veps:\nabla\psi\;=\;
\int^T_0\!\!\int_\Omega V_{\veps,M}\otimes\sqrt{\rho_\veps}\,u_\veps:\nabla\psi\,+\,
\int^T_0\!\!\int_\Omega\bigl(\Id-S_M\bigr)V_\veps\otimes\sqrt{\rho_\veps}\,u_\veps:\nabla\psi \\
& & \qquad\qquad\qquad =\;\int^T_0\!\!\int_\Omega V_{\veps,M}\otimes V_\veps:\nabla\psi\,+\,O(\veps)\,+\,
\int^T_0\!\!\int_\Omega\bigl(\Id-S_M\bigr)V_\veps\otimes\sqrt{\rho_\veps}\,u_\veps:\nabla\psi\,,
\end{eqnarray*}
where, in passing from the first to the second equality, we have argued exactly as before.

Let us focus on the high frequency term first: we can write
\begin{eqnarray*}
\hspace{-0.5cm} \int^T_0\!\!\int_\Omega\bigl(\Id-S_M\bigr)V_\veps\otimes\sqrt{\rho_\veps}\,u_\veps:\nabla\psi & = & 
-\,\veps\int^T_0\!\!\int_\Omega\frac{\rho_\veps-1}{\veps}\,\bigl(\Id-S_M\bigr)V_\veps\otimes\sqrt{\rho_\veps}\,u_\veps:\nabla\psi\,+ \\
& & \qquad\qquad\qquad +\,\int^T_0\!\!\int_\Omega\bigl(\Id-S_M\bigr)V_\veps\otimes\rho_\veps^{3/2}\,u_\veps:\nabla\psi\,.
\end{eqnarray*}
Notice that the former term is $O(\veps)$ (by uniform bounds again). On the other hand, \eqref{sing-b:D(rho-u)} tells us that
$\bigl(\rho_\veps^{3/2}u_\veps\,\nabla\psi\bigr)_\veps$ is uniformly bounded in $L^2_T\bigl(W^{1,3/2}_{loc}(\Omega)\bigr)$: then, since
$\bigl(V_\veps\bigr)_\veps\,\subset\,L^2_T\bigl(L^2\bigr)$ and the operator $\Id-S_M$ is self-adjoint, by Lemma \ref{l:density}
we gather that the latter one is arbitrarly small for $M$ large enough, uniformly in $\veps$. In the end, we obtain that
$$
\lim_{M\ra+\infty}\,\limsup_{\veps\ra0}\,
\left|\int^T_0\!\!\int_\Omega\bigl(\Id-S_M\bigr)V_\veps\otimes\sqrt{\rho_\veps}\,u_\veps:\nabla\psi\right|\,=\,0\,.
$$

It remains us to consider the integral
$$
\int^T_0\!\!\int_\Omega V_{\veps,M}\otimes V_\veps:\nabla\psi\,=\,
\int^T_0\!\!\int_\Omega V_{\veps,M}\otimes V_{\veps,M}:\nabla\psi\,+\,
\int^T_0\!\!\int_\Omega V_{\veps,M}\otimes\bigl(\Id-S_M\bigr)V_\veps:\nabla\psi
$$
and prove that the last term in the right-hand side of the previous relation is small. We notice that, by the decomposition of
Proposition \ref{p:decomp}, we have
$$
\left\|D\bigl(V_{\veps,M}\,:\,\nabla\psi\bigr)\right\|_{L^2_T(L^2_{loc})}\,\leq\,C_1\,+\,\veps\,C_2(M)\,,
$$
for some constant $C_1>0$ independent of $\veps$ and $M$, and $C_2(M)$ just depending on $M$. Therefore, using again the simmetry
of the operator $\Id-S_M$ and Lemma \ref{l:density}, it immediately follows that
$$
\lim_{M\ra+\infty}\,\limsup_{\veps\ra0}\,
\left|\int^T_0\!\!\int_\Omega V_{\veps,M}\otimes\bigl(\Id-S_M\bigr)V_\veps:\nabla\psi\right|\,=\,0\,.
$$
This concludes the proof of the lemma.
\end{proof}

Now, recall  equation \eqref{eq:sing_weak}: paying attention once again to the right signs, by the previous lemma we have just to pass to
the limit in the term
\begin{eqnarray*}
& & \hspace{-1cm}
-\,\int^T_0\!\!\int_\Omega V_{\veps,M}\otimes V_{\veps,M}:\nabla\psi\,=\, 
\int^T_0\!\!\int_\Omega \div\left(V_{\veps,M}\otimes V_{\veps,M}\right)\,\cdot\,\psi \\
& & \qquad\qquad\quad =\,\int^T_0\!\!\int_\Omega \div_h\left(\lan V^h_{\veps,M}\ran\otimes\lan V^h_{\veps,M}\ran\right)\,\cdot\,\psi\,+\,
\int^T_0\!\!\int_\Omega \div_h\left(\lan \wtilde{V}^h_{\veps,M}\otimes\wtilde{V}^h_{\veps,M}\ran\right)\,\cdot\,\psi \\
&  & \qquad\qquad\quad =\,\int^T_0\!\!\int_\Omega\left(\mc{T}^1_{\veps,M}\,+\,\mc{T}^2_{\veps,M}\right)\,\cdot\,\psi\,.
\end{eqnarray*}

For notational convenience, from now on we will generically denote by $\mc{R}_{\veps,M}$ any remainder, i.e. any term satisfying the property
\begin{equation} \label{eq:remainder}
\lim_{M\ra+\infty}\,\limsup_{\veps\ra0}\,\left|\int^T_0\int_\Omega \mc{R}_{\veps,M}\,\cdot\,\psi\,dx\,dt\right|\,=\,0
\end{equation}
for all test functions $\psi\,\in\,\mc{D}\bigl([0,T[\,\times\Omega;\R^3\bigr)\,\cap\,{\rm  Ker}\,\wtilde{\mc{A}}$.

\paragraph{Handling $\mc{T}^1_{\veps,M}$.}

Since we are dealing with  smooth functions, we can integrate by parts: we get
\begin{eqnarray*}
\mc{T}^1_{\veps,M} & = & \div_{\!h}\!\left(\lan V^h_{\veps,M}\ran\otimes\lan V^h_{\veps,M}\ran\right)\;=\;
\div_{\!h}\!\bigl(\lan V^h_{\veps,M}\ran\bigr)\;\lan V^h_{\veps,M}\ran\,+\,
\lan V^h_{\veps,M}\ran\cdot\nabla_h\left(\lan V^h_{\veps,M}\ran\right) \\
& = & \div_{\!h}\bigl(\lan V^h_{\veps,M}\ran\bigr)\;\lan V^h_{\veps,M}\ran\,+\,
\dfrac{1}{2}\,\nabla_h\left(\left|\lan V^h_{\veps,M}\ran\right|^2\right)\,+\,
{\rm curl}_h\lan V^h_{\veps,M}\ran\;\lan V^h_{\veps,M}\ran^\perp\,.
\end{eqnarray*}
Notice that we can forget about the second term: it is a perfect gradient, and then it vanishes when tested against a function in the kernel
of the singular perturbation operator.

For the first term, we take advantage of system \eqref{reg:approx-w}: averaging the first equation with respect to $x^3$
and multiplying it by $\lan V^h_{\veps,M}\ran$, we arrive at
$$
\div_h\bigl(\lan V^h_{\veps,M}\ran\bigr)\;\lan V^h_{\veps,M}\ran\;=\;-\,\veps\,\d_t\lan r_{\veps,M}\ran\,\lan V^h_{\veps,M}\ran\;=\;
\mc{R}_{\veps,M}\,+\,\veps\,\lan r_{\veps,M}\ran\,\d_t\lan V^h_{\veps,M}\ran\,,
$$
since $\veps\,\d_t\bigl(\lan r_{\veps,M}\ran \,\lan V^h_{\veps,M}\ran\bigr)$ is a remainder in the sense specified by relation
\eqref{eq:remainder} above. We use now the horizontal part of \eqref{reg:approx-w} (again, after taking the vertical average),
multiplied by $\lan r_{\veps,M}\ran$: paying attention to  the signs, we get
\begin{eqnarray*}
\veps\,\lan r_{\veps,M}\ran\,\d_t\lan V^h_{\veps,M}\ran & = & -\,\mf{c}(x^h)\,\lan r_{\veps,M}\ran\,\lan V^h_{\veps,M}\ran^\perp\,-\,
\lan r_{\veps_M}\ran\,\nabla_h\bigl(\Id-\Delta_h\bigr)\lan r_{\veps_M}\ran\,+\,\veps\,\lan f^h_{\veps,M}\ran\,+\,\lan g^h_{\veps,M}\ran \\
& = & -\,\mf{c}(x^h)\,\lan r_{\veps,M}\ran\,\lan V^h_{\veps,M}\ran^\perp\,+\,
\bigl(\Id-\Delta_h\bigr)\lan r_{\veps_M}\ran\,\nabla_h\lan r_{\veps_M}\ran\,+\,\mc{R}_{\veps,M}\,,
\end{eqnarray*}
where we used also the properties proved in Proposition \ref{p:regular} and we included in the remainder term also the perfect gradient.
Inserting this relation into the expression for $\mc{T}^1_{\veps,M}$, we find
\begin{equation} \label{eq:T^1_a}
\mc{T}^1_{\veps,M}\,=\,\biggl({\rm curl}_h\lan V^h_{\veps,M}\ran\,-\,\mf{c}(x^h)\,\lan r_{\veps,M}\ran\biggr)\,\lan V^h_{\veps,M}\ran^\perp
\,+\,\bigl(\Id-\Delta_h\bigr)\lan r_{\veps_M}\ran\,\nabla_h\lan r_{\veps_M}\ran\,+\,\mc{R}_{\veps,M}\,.
\end{equation}
In order to deal with the first term in the right-hand side, the idea is to decompose $V^h_{\veps,M}$ in the orthonormal basis
(up to normalization) $\left\{\nabla_h\mf{c}\,,\,\nabla^\perp_h\mf{c}\right\}$. Of course, this can be done in the region when
$\nabla_h\mf{c}$ is far from $0$: therefore, we proceed as follows.

First of all, let us come back to system \eqref{reg:approx-w} for a while: we take again the vertical average of the equations and we
compute the ${\rm curl}$ of the horizontal part, finding
$$
\veps\,\d_t{\rm curl}_h\lan V^h_{\veps,M}\ran\,+\,\div_{\!h}\left(\mf{c}(x^h)\,\lan V^h_{\veps,M}\ran\right)\,=\,
\veps\,{\rm curl}_h\lan f^h_{\veps,M}\ran\,+\,{\rm curl}_h\lan g^h_{\veps,M}\ran\,.
$$
On the other hand, from the first equation multiplied by $\mf{c}$, we get
$$
\veps\,\d_t\left(\mf{c}(x^h)\,\lan r_{\veps,M}\ran\right)\,+\,\div_{\!h}\left(\mf{c}(x^h)\,\lan V^h_{\veps,M}\ran\right)\,=\,
\lan V^h_{\veps,M}\ran\,\cdot\,\nabla_h\mf{c}(x^h)\,,
$$
and this relation, together with the previous one, finally gives us
\begin{equation} \label{eq:curl-cr}
\veps\,\d_t\left({\rm curl}_h\lan V^h_{\veps,M}\ran\,-\,\mf{c}(x^h)\,\lan r_{\veps,M}\ran\right)\,=\,
\veps\,{\rm curl}_h\lan f^h_{\veps,M}\ran\,+\,{\rm curl}_h\lan g^h_{\veps,M}\ran\,+\,\lan V^h_{\veps,M}\ran\,\cdot\,\nabla_h\mf{c}(x^h)\,.
\end{equation}
Notice that, thanks to Proposition \ref{p:regular}, there exists a function $\eta\geq0$, with $\eta(M)\longrightarrow0$ for $M\ra+\infty$,
such that, for any compact $K\,\subset\,\Omega$,
\begin{equation} \label{est:curl_rem}
\sup_{\veps>0}\left\|{\rm curl}_h\lan g^h_{\veps,M}\ran\right\|_{L^2_T\bigl(L^2(K)\bigr)}
\,\leq\,\eta(M)\,.
\end{equation}
Then, fixed a $b\in\mc{C}^{\infty}_0(\R^2)$, with $0\leq b(x^h)\leq1$, such that $b\equiv1$ on $\left\{|x^h|\leq1\right\}$ and
$b\equiv0$ on $\left\{|x^h|\geq2\right\}$, we define
$$
b_M(x^h)\,:=\,b\left(\bigl(\eta(M)\bigr)^{\!-1/2}\,\nabla_h\mf{c}(x^h)\right)\,.
$$

Now we are ready to deal with the first term in the right-hand side of \eqref{eq:T^1_a}. For notational convenience, we set
$\mc{X}_{\veps,M}\,:=\,{\rm curl}_h\lan V^h_{\veps,M}\ran\,-\,\mf{c}(x^h)\,\lan r_{\veps,M}\ran$.
On the one hand, using the decomposition
and the bounds established in Proposition \ref{p:decomp}, we deduce that, for any compact $K\subset\Omega$,
\begin{eqnarray*}
\left\|b_M\,\mc{X}_{\veps,M}\,\lan V^h_{\veps,M}\ran^\perp\right\|_{L^1([0,T]\times K)} & \leq &
\veps\,C(M)\,+\,C\,\left\|b_M\right\|_{L^6(K)} \\
& \leq & 
\veps\,C(M)\,+\,C\,\left(\mc{L}\left\{x^h\in\R^2\;\bigl|\;\left|\nabla_h\mf{c}(x^h)\right|\,\leq\,2\,\sqrt{\eta(M)}\right\}\right)^{\!1/6}\,.
\end{eqnarray*}
Therefore, thanks to hypothesis \eqref{eq:non-crit}, we infer that this term is a remainder, in the sense specified by
relation \eqref{eq:remainder}.
On the other hand, for $\nabla_h\mf{c}$ far from $0$, we can write
\begin{eqnarray*}
\left(1-b_M\right)\,\mc{X}_{\veps,M}\,\lan V^h_{\veps,M}\ran^\perp & = & \left(1-b_M\right)\,\mc{X}_{\veps,M}
\left(\frac{\lan V^h_{\veps,M}\ran^\perp\cdot\nabla_h^\perp\mf{c}}{\left|\nabla_h\mf{c}\right|^2}\,\nabla_h^\perp\mf{c}\,+\,
\frac{\lan V^h_{\veps,M}\ran^\perp\cdot\nabla_h\mf{c}}{\left|\nabla_h\mf{c}\right|^2}\,\nabla_h\mf{c}\right) \\
& = & \left(1-b_M\right)\,\mc{X}_{\veps,M}
\left(\frac{\lan V^h_{\veps,M}\ran\cdot\nabla_h\mf{c}}{\left|\nabla_h\mf{c}\right|^2}\,\nabla_h^\perp\mf{c}\,+\,
\frac{\lan V^h_{\veps,M}\ran^\perp\cdot\nabla_h\mf{c}}{\left|\nabla_h\mf{c}\right|^2}\,\nabla_h\mf{c}\right)\,.
\end{eqnarray*}
We observe that the latter term in the right-hand side is identically $0$ when tested against a $\psi\in{\rm Ker}\,\wtilde{A}$
(because $\div_{\!h}\psi^h\,=\,\div_{\!h}\left(\mf{c}\psi^h\right)\equiv0$, and then $\psi^h\cdot\nabla_h\mf{c}=0$).
For the former term, instead, we use the expression found in \eqref{eq:curl-cr}: we get
\begin{eqnarray*}
& & \hspace{-0.7cm} \left(1-b_M\right)\,\mc{X}_{\veps,M}\frac{\lan V^h_{\veps,M}\ran\cdot\nabla_h\mf{c}}{\left|\nabla_h\mf{c}\right|^2}\,
\nabla_h^\perp\mf{c}\;= \\
& & \qquad =\;\frac{\left(1-b_M\right)\,\mc{X}_{\veps,M}}{\left|\nabla_h\mf{c}\right|^2}\,\veps\,\d_t\mc{X}_{\veps,M}\,
\nabla_h^\perp\mf{c}\,-\,\frac{\left(1-b_M\right)\,\mc{X}_{\veps,M}}{\left|\nabla_h\mf{c}\right|^2}
\left(\veps\,{\rm curl}_h\lan f^h_{\veps,M}\ran\,+\,{\rm curl}_h\lan g^h_{\veps,M}\ran\right)\,\nabla_h^\perp\mf{c} \\
& & \qquad =\;\frac{\veps\,\left(1-b_M\right)\,\d_t\left|\mc{X}_{\veps,M}\right|^2}{2\,\left|\nabla_h\mf{c}\right|^2}\,\nabla_h^\perp\mf{c}
\,-\,\frac{\left(1-b_M\right)\,\mc{X}_{\veps,M}}{\left|\nabla_h\mf{c}\right|^2}
\left(\veps\,{\rm curl}_h\lan f^h_{\veps,M}\ran\,+\,{\rm curl}_h\lan g^h_{\veps,M}\ran\right)\,\nabla_h^\perp\mf{c}\,,
\end{eqnarray*}
which is again a remainder $\mc{R}_{\veps,M}$, thanks to Proposition \ref{p:decomp} and property \eqref{est:curl_rem}.


In the end, putting all these facts together, we have proved that (paying attention again to the right signs)
\begin{equation} \label{eq:T^1_final}
\mc{T}^1_{\veps,M}\,=\,\bigl(\Id-\Delta_h\bigr)\lan r_{\veps_M}\ran\,\nabla_h\lan r_{\veps_M}\ran\,+\,\mc{R}_{\veps,M}\,.
\end{equation}
Notice that the first term in the right-hand side exactly cancels out with the first one of \eqref{comp-cpt:dens}.

\paragraph{Dealing with $\mc{T}^2_{\veps,M}$.} 
Let us now consider the term $\mc{T}^2_{\veps,M}$: exactly as done above, we can write
$$
\mc{T}^2_{\veps,M}\,=\,\div_{\!h}\left(\lan \wtilde{V}^h_{\veps,M}\otimes \wtilde{V}^h_{\veps,M}\ran\right)
\,=\,\lan \div_{\!h}\bigl(\wtilde{V}^h_{\veps,M}\bigr)\;\wtilde{V}^h_{\veps,M}\ran\,+\,
\dfrac{1}{2}\,\lan\nabla_h\left|\wtilde{V}^h_{\veps,M}\right|^2\ran\,+\,
\lan {\rm curl}_h\wtilde{V}^h_{\veps,M}\;\left(\wtilde{V}^h_{\veps,M}\right)^\perp\ran\,.
$$

Let us focus on the last term for a while: with the notations introduced in \eqref{dec:vert-av}, we have
$$
\left({\rm curl}\wtilde{V}_{\veps,M}\right)^h\,=\,\d_3\wtilde{W}^h_{\veps,M}\qquad\qquad\mbox{ and }\qquad\qquad
\left({\rm curl}\wtilde{V}_{\veps,M}\right)^3\,=\,{\rm curl}_h\wtilde{V}^h_{\veps,M}\,=\,\wtilde{\omega}^3_{\veps,M}\,,
$$
where we have defined $\wtilde{W}^h_{\veps,M}\,:=\,\left(\wtilde{V}^h_{\veps,M}\right)^\perp\,-\,
\d_3^{-1}\nabla^\perp_h\wtilde{V}^3_{\veps,M}$.
For these quantities, from the momentum equation in \eqref{reg:approx-w} (where we take the mean-free part and the ${\rm curl}$),
we deduce
\begin{equation} \label{eq:mean-free}
\begin{cases}
\veps\,\d_t\wtilde{W}^h_{\veps,M}\,-\,\mf{c}\,\wtilde{V}^h_{\veps,M}\,=\,\left(\d_3^{-1}\,
{\rm curl}\left(\veps\,\wtilde{f}_{\veps,M}\,+\,\wtilde{g}_{\veps,M}\right)\right)^h \\[1ex]
\veps\,\d_t\wtilde{\omega}^3_{\veps,M}\,+\,\div_{\!h}\bigl(\mf{c}\,\wtilde{V}^h_{\veps,M}\bigr)\,=\,{\rm curl}_h\left(\veps\,\wtilde{f}^h_{\veps,M}\,+\,
\wtilde{g}^h_{\veps,M}\right)
\end{cases}
\end{equation}
Making use of the relations above and of Propositions \ref{p:regular} and \ref{p:decomp}, we get
\begin{eqnarray*}
{\rm curl}_h\wtilde{V}^h_{\veps,M}\;\left(\wtilde{V}^h_{\veps,M}\right)^\perp & = &
\wtilde{\omega}^3_{\veps,M}\,\left(\wtilde{V}^h_{\veps,M}\right)^\perp \\
& = & \frac{\veps}{\mf c}\,\d_t\left(\wtilde{W}^h_{\veps,M}\right)^\perp\,\wtilde{\omega}^3_{\veps,M}\,-\,\frac{1}{\mf c}\,\wtilde{\omega}^3_{\veps,M}\,
\left(\left(\d_3^{-1}\,{\rm curl}\left(\veps\,\wtilde{f}_{\veps,M}\,+\,\wtilde{g}_{\veps,M}\right)\right)^h\right)^\perp \\
& = & -\,\frac{\veps}{\mf c}\,\left(\wtilde{W}^h_{\veps,M}\right)^\perp\,\d_t\wtilde{\omega}^3_{\veps,M}\,+\,\mc{R}_{\veps,M}\;=\;
\frac{1}{\mf c}\,\left(\wtilde{W}^h_{\veps,M}\right)^\perp\,\div_{\!h}\bigl(\mf{c}\,\wtilde{V}^h_{\veps,M}\bigr)\,+\,\mc{R}_{\veps,M}\,.
\end{eqnarray*}

Hence, including also the gradient term into the remainders, we arrive at the equality
$$ 
\mc{T}^2_{\veps,M}\,=\,\lan \div_{\!h}\bigl(\wtilde{V}^h_{\veps,M}\bigr)\,
\left(\wtilde{V}^h_{\veps,M}\,+\,\left(\wtilde{W}^h_{\veps,M}\right)^\perp\right)\ran\,+\,
\lan \frac{1}{\mf c}\,\left(\wtilde{W}^h_{\veps,M}\right)^\perp\,\wtilde{V}^h_{\veps,M}\cdot\nabla_h\mf{c} \ran\,+\,\mc{R}_{\veps,M}\,,
$$
which can be finally rewritten in the following way:
\begin{eqnarray*}
\mc{T}^2_{\veps,M} & = & \lan \div\wtilde{V}_{\veps,M}\,\left(\wtilde{V}^h_{\veps,M}\,+\,\left(\wtilde{W}^h_{\veps,M}\right)^\perp\right)\ran\,- \\
& & \qquad\qquad -\,
\lan \d_3\wtilde{V}^3_{\veps,M}\,\left(\wtilde{V}^h_{\veps,M}\,+\,\left(\wtilde{W}^h_{\veps,M}\right)^\perp\right)\ran\,+\,
\lan \frac{1}{\mf c}\,\left(\wtilde{W}^h_{\veps,M}\right)^\perp\,\wtilde{V}^h_{\veps,M}\cdot\nabla_h\mf{c} \ran\,+\,\mc{R}_{\veps,M}\,.
\end{eqnarray*}
The second term on the right-hand side is actually another remainder. As a matter of fact, direct computations yield
\begin{eqnarray*}
\d_3\wtilde{V}^3_{\veps,M}\left(\wtilde{V}^h_{\veps,M}+\left(\wtilde{W}^h_{\veps,M}\right)^\perp\right) & = & 
\d_3\!\!\left(\wtilde{V}^3_{\veps,M}\left(\wtilde{V}^h_{\veps,M}+\left(\wtilde{W}^h_{\veps,M}\right)^\perp\right)\right)-
\wtilde{V}^3_{\veps,M}\,\d_3\!\left(\wtilde{V}^h_{\veps,M}+\left(\wtilde{W}^h_{\veps,M}\right)^\perp\right) \\
& = & \mc{R}_{\veps,M}\,-\,\frac{1}{2}\,\nabla_h\left|\wtilde{V}^3_{\veps,M}\right|^2\;=\;\mc{R}_{\veps,M}\,.
\end{eqnarray*}
As for the first term, instead, we use the equation for the density in \eqref{reg:approx-w} to obtain
\begin{eqnarray*}
\div\wtilde{V}_{\veps,M}\left(\wtilde{V}^h_{\veps,M}+\left(\wtilde{W}^h_{\veps,M}\right)^\perp\right) & = & 
-\,\veps\,\d_t\wtilde{r}_{\veps,M}\left(\wtilde{V}^h_{\veps,M}+\left(\wtilde{W}^h_{\veps,M}\right)^\perp\right) \\
& = & \mc{R}_{\veps,M}\,+\,\veps\,\wtilde{r}_{\veps,M}\;\d_t\!\left(\wtilde{V}^h_{\veps,M}+\left(\wtilde{W}^h_{\veps,M}\right)^\perp\right)\,.
\end{eqnarray*}
Now, by equations \eqref{eq:mean-free} and \eqref{reg:approx-w} again, it is easy to see that
$$
\veps\,\wtilde{r}_{\veps,M}\;\d_t\!\left(\wtilde{V}^h_{\veps,M}+\left(\wtilde{W}^h_{\veps,M}\right)^\perp\right)\,=\,
\mc{R}_{\veps,M}\,-\,\wtilde{r}_{\veps,M}\,\nabla\bigl(\Id-\Delta\bigr)\wtilde{r}_{\veps,M}\,=\,
\mc{R}_{\veps,M}\,+\,\nabla\wtilde{r}_{\veps,M}\,\bigl(\Id-\Delta\bigr)\wtilde{r}_{\veps,M}\,,
$$
and therefore we find (with attention to the right sign)
$$
\mc{T}^2_{\veps,M}\,=\,\lan \nabla\wtilde{r}_{\veps,M}\,\bigl(\Id-\Delta\bigr)\wtilde{r}_{\veps,M}\ran\,+\,
\lan \frac{1}{\mf c}\,\left(\wtilde{W}^h_{\veps,M}\right)^\perp\,\wtilde{V}^h_{\veps,M}\cdot\nabla_h\mf{c} \ran\,+\,\mc{R}_{\veps,M}\,. 
$$

Now we have to deal with the second term in this last identity. Once again, we take advantage of the decomposition along the basis $\left\{\nabla_h\mf{c}\,,\,\nabla^\perp_h\mf{c}\right\}$.
For simplicity of exposition, we omit the cut-off away from the region $\{\nabla_h\mf c=0\}$ and we just give a sketch of the argument, since it
is analogous to what done above for $\mc{T}^1_{\veps,M}$.

First of all, we write
$$
\frac{1}{\mf c}\,\left(\wtilde{W}^h_{\veps,M}\right)^\perp\,\wtilde{V}^h_{\veps,M}\cdot\nabla_h\mf{c}\,=\,\frac{1}{\mf c}\left(\wtilde{V}^h_{\veps,M}\cdot\nabla_h\mf{c}\right)
\left(\wtilde{W}^h_{\veps,M}\cdot\nabla_h\mf{c}\,\frac{\nabla_h^\perp\mf{c}}{\left|\nabla_h\mf{c}\right|^2}\,+\,\left(\wtilde{W}^h_{\veps,M}\right)^\perp\cdot\nabla_h\mf{c}\,
\frac{\nabla_h\mf{c}}{\left|\nabla_h\mf{c}\right|^2}\right)\,.
$$
As before, the last term in the right-hand side vanishes when tested against a smooth $\psi\in{\rm Ker}\,\wtilde{\mc A}$.
Next, we obtain informations on $\wtilde{V}^h_{\veps,M}\cdot\nabla_h\mf{c}$ from the first equation in \eqref{eq:mean-free}:
$$
\wtilde{V}^h_{\veps,M}\cdot\nabla_h\mf{c}\,=\,\frac{1}{\mf c}\left(\veps\,\d_t\wtilde{W}^h_{\veps,M}\,-\,\left(\d_3^{-1}\,
{\rm curl}\left(\veps\,\wtilde{f}_{\veps,M}\,+\,\wtilde{g}_{\veps,M}\right)\right)^h\right)\cdot\nabla_h\mf{c}\,.
$$
Therefore, we obtain that
\begin{eqnarray*}
\frac{1}{\mf c}\,\left(\wtilde{W}^h_{\veps,M}\right)^\perp\,\wtilde{V}^h_{\veps,M}\cdot\nabla_h\mf{c} & = & \frac{\veps}{2\,\mf c^2}\,
\d_t\left|\wtilde{W}^h_{\veps,M}\cdot\nabla_h\mf{c}\right|^2\,\frac{\nabla_h^\perp\mf{c}}{\left|\nabla_h\mf{c}\right|^2}\,- \\
& & \qquad-\,\frac{1}{\mf c^2}\,\left(\d_3^{-1}\,{\rm curl}\left(\veps\,\wtilde{f}_{\veps,M}\,+\,\wtilde{g}_{\veps,M}\right)\right)^h\cdot\nabla_h\mf{c}\;
\frac{\nabla_h^\perp\mf{c}}{\left|\nabla_h\mf{c}\right|^2}\,,
\end{eqnarray*}
which is obviously a remainder in the sense of relation \eqref{eq:remainder}.

In the end, we have discovered that (paying attention to the right sign)
\begin{equation} \label{eq:T^2_final}
\mc{T}^2_{\veps,M}\,=\,\lan \nabla\wtilde{r}_{\veps,M}\,\bigl(\Id-\Delta\bigr)\wtilde{r}_{\veps,M}\ran\,+\,\mc{R}_{\veps,M}\,. 
\end{equation}
Notice that the density-dependent term exactly cancels out with the latter item in \eqref{comp-cpt:dens}.

\subsubsection{The limit equation}

Let us sum up what we have just proved. In order to pass to the limit in equation \eqref{eq:sing_weak}, we needed to treat the non-linearities coming
from the capillarity term and the convection term.

For the former, after applying Lemma \ref{l:capill_approx}, we have arrived at relation \eqref{comp-cpt:dens}:
\begin{eqnarray*}
-\,\int^T_0\!\!\int_\Omega\bigl(\Id-\Delta\bigr)r_{\veps,M}\;\nabla r_{\veps,M}\cdot\psi & = & 
-\,\int^T_0\!\!\int_\Omega\bigl(\Id-\Delta_h\bigr)\lan r_{\veps,M}\ran\;\nabla_h\lan r_{\veps,M}\ran\cdot\psi\,- \\
& & \qquad\qquad
-\,\int^T_0\!\!\int_\Omega\lan\bigl(\Id-\Delta\bigr)\wtilde{r}_{\veps,M}\;\nabla\wtilde{r}_{\veps,M}\ran\cdot\psi\,, \nonumber
\end{eqnarray*}
On the other hand, for the latter we exploited Lemma \ref{l:conv_approx} and, after manipulations, we have found
\begin{eqnarray*}
& & \hspace{-1cm}
-\,\int^T_0\!\!\int_\Omega V_{\veps,M}\otimes V_{\veps,M}:\nabla\psi\,=\,\int^T_0\!\!\int_\Omega\left(\mc{T}^1_{\veps,M}\,+\,\mc{T}^2_{\veps,M}\right)\,\cdot\,\psi \\
& & \qquad=\,\int^T_0\!\!\int_\Omega\biggl(\bigl(\Id-\Delta_h\bigr)\lan r_{\veps_M}\ran\,\nabla_h\lan r_{\veps_M}\ran\,+\,
\lan \nabla\wtilde{r}_{\veps,M}\,\bigl(\Id-\Delta\bigr)\wtilde{r}_{\veps,M}\ran\,+\,\mc{R}_{\veps,M}\biggr)\cdot\,\psi\,,
\end{eqnarray*}
where we used also relations \eqref{eq:T^1_final} and \eqref{eq:T^2_final}.
Therefore, thanks to the special cancellations involving the density-dependent terms, in the end we have proved the relation
$$
\int^T_0\!\!\int_\Omega\biggl(-\,\bigl(\Id-\Delta\bigr)r_{\veps,M}\;\nabla r_{\veps,M}\cdot\psi\,-\,V_{\veps,M}\otimes V_{\veps,M}:\nabla\psi\biggr)\,=\,
\int^T_0\!\!\int_\Omega\mc{R}_{\veps,M}\cdot\psi\,,
$$
which immediately implies, together with Lemmas \ref{l:capill_approx} and \ref{l:conv_approx}, that
$$
\lim_{M\ra+\infty}\,\lim_{\veps\ra0}
\int^T_0\!\!\int_\Omega\biggl(\frac{1}{\veps^2}\,\Delta\rho_\veps\,\nabla\rho_\veps\cdot\psi\,-\,\rho_\veps\,u_\veps\otimes u_\veps:\nabla\psi\biggr)dx\,dt\,=\,0\,.
$$

Then, thanks to the previous computations, we can pass to the limit in the weak formulation of our system: we obtain
$$
\int^T_0\!\!\int_\Omega\biggl(-u\cdot\d_t\psi-r\d_t\bigl(\Id-\Delta_h\bigr)\phi+\nu Du:\nabla\psi\biggr)dxdt\,=\,
\int_\Omega\bigl(u_0\cdot\psi(0)+r_0\bigl(\Id-\Delta_h\bigr)\phi(0)\bigr)dx
$$
for any $(\phi,\psi)$ test functions belonging to the kernel of the singular perturbation operator $\wtilde{\mc{A}}$.
Recall that, in particular, this implies the relation $\mf{c}\,\psi^h\,=\,\nabla_h^\perp\bigl(\Id-\Delta_h\bigr)\phi$.

Furthermore, we recall that also $(r,u)\,\in\,{\rm Ker}\,\wtilde{\mc{A}}$: then $\div u\equiv0$, $u=\bigl(u^h,0\bigr)$ and
$\mf{c}\,u^h\,=\,\nabla_h^\perp\bigl(\Id-\Delta_h\bigr)r$.

Setting $X(r)\,=\,\bigl(\Id-\Delta_h\bigr)r$ and $\wtilde{\phi}\,=\,\bigl(\Id-\Delta_h\bigr)\phi$, and
using that all the functions depend just on the horizontal variables, straightforward computations yield to
\begin{eqnarray*}
& & \hspace{-0.7cm} -\int^T_0\int_\Omega u\cdot\d_t\psi\,dx\,dt\,=\,
\int^T_0\int_{\R^2}\div_{\!h}\!\left(\frac{1}{\mf{c}^2}\,\nabla_hX(r)\right)\,\d_t\wtilde{\phi}\,dx^h\,dt \\
& & \hspace{-1cm} \nu\int^T_0\int_\Omega Du:\nabla\psi\,dx\,dt\,=\,
\nu\int^T_0\int_{\R^2}\mf{D}_{\mf{c}}\bigl(X(r)\bigr)\,:\,\nabla_h\bigl(\mf{c}^{-1}\,\nabla^\perp_h\wtilde{\phi}\,\bigr)\,dx^h\,dt \\
& & \qquad\qquad =\;\nu\int^T_0\int_{\R^2}\mf{D}_{\mf{c}}\bigl(X(r)\bigr)\,:\,\mf{D}_{\mf{c}}\bigl(\,\wtilde{\phi}\,\bigr)\,dx^h\,dt 
\,=\,\nu\int^T_0\int_{\R^2}\,^t\mf{D}_{\mf{c}}\,\circ\,\mf{D}_{\mf{c}}\bigl(X(r)\bigr)\,\wtilde{\phi}\,dx^h\,dt\,.
\end{eqnarray*}

Inserting these equalities into the previous relation completes the proof of Theorem \ref{t:sing_var}.

\subsection{Final remarks} \label{ss:regularity}

We conclude by making a few comments about our hypotheses on the regularity of the rotation coefficient, namely
$\mf{c}\in W^{1,\infty}(\R^2)$ with $\nabla_h\mf{c}\in\mc{C}_\mu(\R^2)$.

\begin{itemize}
\item At the level of uniform bounds (see Section \ref{s:bounds}), we asked for $\nabla_h\mf{c}\in L^\infty$ in order to close the
estimates for the BD-entropy (recall Lemma \ref{l:rot}), both in the case of vanishing and constant capillarity.
Nonetheless notice that, under an additional $L^2_T\bigl(L^2\bigr)$ bounds for the family of velocity fields
$\left(u_\veps\right)_\veps$, the control of the rotation term in \eqref{est:F} is straightforward for $\mf{c}\in L^\infty$:
$$
\frac{1}{\veps}\left|\int^t_0\!\!\int_\Omega\mf{c}(x^h)\,e^3\times u_\veps\cdot\nabla\rho_\veps\,dx\,d\tau\right|\,\leq\,
C\,\sqrt{T}\;\|\mf{c}\|_{L^\infty}\,\left\|u_\veps\right\|_{L^2_T(L^2)}\,\frac{\left\|\nabla\rho_\veps\right\|_{L^\infty_T(L^2)}}{\veps}\,,
$$
where the right-hand side is uniformly bounded by classical energy estimates (see Corollary \ref{c:E}). Remark that
the property $\left(u_\veps\right)_\veps\,\subset\,L^2_T\bigl(L^2\bigr)$ can be deduced, for instance, when an additional
laminar friction term is added to the momentum equation (see e.g. \cite{B-D_2003}).
\item The $W^{1,\infty}$ regularity seems to be necessary in the compensated compactness argument: in particular, we exploited it for resorting to
the equation for the vorticity and for decomposing some vector fields along
$\left\{\nabla_h\mf{c}\,,\,\nabla_h^\perp\mf{c}\right\}$. \\
Notice that, by Proposition \ref{p:TP}, $\nabla_h\mf{c}$ is parallel to $\nabla_h(\Id-\Delta_h)r$;
on the other hand, we have no available informations for $\nabla_h(\Id-\Delta_h)r$: in particular, it is not clear for us how to avoid
the non-degeneracy condition \eqref{eq:non-crit}.
\item Finally, the requirement $\nabla_h\mf{c}\in\mc{C}_\mu$ (for some admissible modulus of continuity $\mu$) is fundamental in order to
get \eqref{est:curl_rem}, which is a key property in our proof. Notice that paraproduct decomposition does not seem to help in gaining
anything at this level. This is the main reason why imposing some regularity on
$\nabla_h\mf{c}$ seems to be necessary for proving the result, and it seems not to be possible to go below the Lipschitz threshold
$\mf{c}\in W^{1,\infty}$.
\end{itemize}

Related to the non-degeneracy condition, let us remark also that there is a huge gap between the case of constant rotation axis
(for which $\nabla\mf{c}\equiv0$ everywhere) and the case of variable rotation axis considered here, for which
we need hypothesis \eqref{eq:non-crit}.
It seems very likely that this difference is just ``artificial'', and it depends on the techniques used in the proof.
It could be interesting to find a different approach to the problem, in order to being able to treat coefficients which can be constant (i.e. whose gradient
can vanish) in a region of non-zero Lebesgue measure.

\appendix

\section{Appendix} \label{app}

We collect in this appendix some well-known results about Littlewood-Paley theory and admissible moduli of continuity, and the proof of
some technical lemmas.

\subsection{Fourier Analysis toolbox} \label{app:LP}

Let us recall some tools from Fourier Analysis, which we exploited in our proof.
We refer e.g. to \cite{B-C-D} (Chapter 2) and \cite{M-2008} (Chapters 4 and 5) for the details.

For simplicity of exposition, let us deal with the $\R^d$ case; however, the construction can be adapted to the $d$-dimensional
torus $\mbb{T}^d$, and then also to the case of $\R^{d_1}\times\mbb{T}^{d_2}$.

\medbreak
First of all, let us introduce the \emph{Littlewood-Paley decomposition}, based on a non-homogeneous dyadic partition of unity in the
Phase Space.

We fix a smooth radial function $\chi$ supported in the ball $B(0,2)$,  equal to $1$ in a neighborhood of $B(0,1)$
and such that $r\mapsto\chi(r\,e)$ is nonincreasing over $\R_+$ for all unitary vectors $e\in\R^d$. Set
$\varphi\left(\xi\right)=\chi\left(\xi\right)-\chi\left(2\xi\right)$ and $\vphi_j(\xi):=\vphi(2^{-j}\xi)$ for all $j\geq0$.

The dyadic blocks $(\Delta_j)_{j\in\Z}$ are defined by\footnote{Throughout we agree  that  $f(D)$ stands for 
the pseudo-differential operator $u\mapsto\mc{F}^{-1}(f\,\mc{F}u)$.} 
$$
\Delta_j:=0\ \hbox{ if }\ j\leq-2,\quad\Delta_{-1}:=\chi(D)\qquad\mbox{ and }\qquad
\Delta_j:=\varphi(2^{-j}D)\; \mbox{ if }\;  j\geq0\,.
$$
We  also introduce the low frequency cut-off operators: for any $j\geq0$,
\begin{equation} \label{eq:low-freq}
S_ju\;:=\;\chi\bigl(2^{-j}D\bigr)u\;=\;\sum_{k\leq j-1}\Delta_{k}u\,.
\end{equation}

The following classical property holds true:
for any $u\in\mc{S}'$, one has the equality $u=\sum_{j}\Delta_ju$ in the sense of $\mc{S}'$.
We will freely use this fact in the sequel.

Spectrally localized functions have nice properties with respect to the action of derivatives. This is explained by the so-called
\emph{Bernstein's inequalities}.
  \begin{lemma} \label{l:bern}
Let  $0<r<R$. A constant $C$ exists so that, for any nonnegative integer $k$, any couple $(p,q)$ 
in $[1,+\infty]^2$ with  $p\leq q$ 
and any function $u\in L^p$,  we  have, for all $\lambda>0$,
$$
\displaylines{
{\rm supp}\, \widehat u \subset   B(0,\lambda R)\quad
\Longrightarrow\quad
\|\nabla^k u\|_{L^q}\, \leq\,
 C^{k+1}\,\lambda^{k+d\left(\frac{1}{p}-\frac{1}{q}\right)}\,\|u\|_{L^p}\;;\cr
{\rm supp}\, \widehat u \subset \{\xi\in\R^d\,|\, r\lambda\leq|\xi|\leq R\lambda\}
\quad\Longrightarrow\quad C^{-k-1}\,\lambda^k\|u\|_{L^p}\,
\leq\,
\|\nabla^k u\|_{L^p}\,
\leq\,
C^{k+1} \, \lambda^k\|u\|_{L^p}\,.
}$$
\end{lemma}   

By use of Littlewood-Paley decomposition, we can define the class of Besov spaces.
\begin{defin} \label{d:B}
  Let $s\in\R$ and $1\leq p,r\leq+\infty$. The \emph{non-homogeneous Besov space}
$B^{s}_{p,r}$ is defined as the subset of tempered distributions $u$ for which
$$
\|u\|_{B^{s}_{p,r}}\,:=\,
\left\|\left(2^{js}\,\|\Delta_ju\|_{L^p}\right)_{j\in\N}\right\|_{\ell^r}\,<\,+\infty\,.
$$
\end{defin}

Besov spaces are interpolation spaces between the Sobolev ones. In fact, for any $k\in\N$ and $p\in[1,+\infty]$
we have the chain of continuous embeddings $B^k_{p,1}\hookrightarrow W^{k,p}\hookrightarrow B^k_{p,\infty}$.
Moreover, for all $s\in\R$ we have $B^s_{2,2}\equiv H^s$, with equivalence of the respective norms.

We report here some statements which we used in our analysis: their proof are given in \cite{F}. We refer again to \cite{B-C-D}
and \cite{M-2008} for more details and more general results.
\begin{lemma} \label{l:density}
\begin{itemize}
 \item[(i)] For $1\leq p\leq2$, one  has $\|f\|_{L^2}\,\leq\,C\bigl(\|f\|_{L^p}\,+\,\|\nabla f\|_{L^2}\bigr)$.
\item[(ii)] For any $0<\delta\leq1/2$ 
and any $1\leq p\leq +\infty$, one has
$$
\|f\|_{L^\infty}\,\leq\,C\left(\|f\|_{L^p}\,+\,\|\nabla f\|^{(1/2)-\delta}_{L^2}\,
\left\|\nabla^2f\right\|^{(1/2)+\delta}_{L^2}\,\right).
$$
\item[(iii)] Let $1\leq p\leq2$ such that $1/p\,<\,1/d\,+\,1/2$. For any $j\in\N$, there exists a constant $C_j$, depending just on
$j$, $d$ and $p$, and going to $0$ for $j\ra+\infty$, such that
$$
\left\|\left(\Id\,-\,S_j\right)f\right\|_{L^2}\,\leq\,C_j\,\left\|\nabla f\right\|_{B^0_{p,\infty}}\,.
$$
In particular, if $\nabla f=\nabla f_1 + \nabla f_2$, with $\nabla f_1\in B^0_{2,\infty}$ and $\nabla f_2\in B^0_{p,\infty}$, then
$$
\left\|\left(\Id\,-\,S_j\right)f\right\|_{L^2}\,\leq\,\wtilde{C}_j\,\left(\left\|\nabla f_1\right\|_{B^0_{2,\infty}}\,+\,
\left\|\nabla f_2\right\|_{B^0_{p,\infty}}\right)\,,
$$
for a new constant $\wtilde{C}_j$ still going to $0$ for $j\ra+\infty$.
\end{itemize}
\end{lemma}

Finally, let us recall some notions of \emph{homogeneous} dyadic decomposition. Namely, one can rather work with homogeneous
dyadic blocks $(\dot\Delta_j)_{j\in\Z}$, defined as
$$
\dot\Delta_j:=\varphi(2^{-j}D)\qquad \mbox{ for all }\quad  j\in\Z\,.
$$
Then, we can introduce the homogeneous Besov spaces $\dot{B}^s_{p,r}$ by the property
$$
\|u\|_{\dot{B}^{s}_{p,r}}\,:=\,
\left\|\left(2^{js}\,\|\dot\Delta_ju\|_{L^p}\right)_{\!j\in\Z}\,\right\|_{\ell^r}\,<\,+\infty
$$
(plus some conditions on low frequencies).
We do not enter into the details, for which we refer to Chapter 2 of \cite{B-C-D}. Let us however recall refined embeddings of
homogeneous Besov spaces into Lebesgue spaces (see Theorem 2.40 of \cite{B-C-D}).
\begin{prop} \label{p:emb_hom-besov}
For any $p\in[2,+\infty]$, one has the embeddings $\dot{B}^0_{p,2}\,\hra\,L^p$ and $L^{p'}\,\hra\,\dot{B}^0_{p',2}$,
where $p'$ is defined by the condition $1/p'=1-1/p$.
\end{prop}

\subsection{Admissible moduli of continuity} \label{app:continuity}

In this paragraph we recall some basic definitions and properties about general moduli of continuity.
We refer to Section 2.11 of \cite{B-C-D} for a more indeep discussion.

\begin{defin} \label{d:mod-cont}
A \emph{modulus of continuity} is a continuous non-decreasing function $\mu:[0,1]\,\longrightarrow\,\R_+$ such that $\mu(0)=0$.

It is said to be \emph{admissible} if the function $\Gamma_\mu$, defined by the relation
$$
\Gamma_\mu(s)\,:=\,s\,\mu(1/s)\,,
$$
is non-decreasing on $[1,+\infty[\,$ and it verifies, for some constant $C>0$ and any $s\geq1$,
$$
\int_s^{+\infty}\sigma^{-2}\,\Gamma_\mu(\sigma)\,d\sigma\,\leq\,C\,s^{-1}\,\Gamma_\mu(s)\,.
$$
\end{defin}

Given a modulus of continuity $\mu$, we can define the space $\mc{C}_\mu(\R^d)$ as the set of real-valued functions $a\in L^\infty(\R^d)$
such that
$$
|a|_{\mc{C}_\mu}\,:=\,\sup_{|y|\in\,]0,1]}\frac{\left|a(x+y)\,-\,a(x)\right|}{\mu(|y|)}\,<\,+\infty\,.
$$
We also define $\|a\|_{\mc{C}_\mu}\,:=\,\|a\|_{L^\infty}\,+\,|a|_{\mc{C}_\mu}$.

On the other hand, for an increasing $\Gamma$ on $[1,+\infty[\,$, we define the space $B_\Gamma(\R^d)$ as the set
of real-valued functions $a\in L^\infty(\R^d)$ such that
$$
|a|_{B_\Gamma}\,:=\,\sup_{j\geq0}\frac{\left\|\nabla S_ja\right\|_{L^\infty}}{\Gamma(2^j)}\,<\,+\infty\,,
$$
where $S_j$ is the low-frequency cut-off operator of a Littlewood-Paley decomposition, as introduced above.
We also set $\|a\|_{B_\Gamma}\,:=\,\|a\|_{L^\infty}\,+\,|a|_{B_\Gamma}$.

One has the following result (see Proposition 2.111 of \cite{B-C-D}).
\begin{prop} \label{p:cont-equiv}
Let $\mu$ be an admissible modulus of continuity. Then $\mc{C}_\mu(\R^d)\,=\,B_{\Gamma_\mu}(\R^d)$,
and the respective norms are equivalent. Moreover, for any $a\in\mc{C}_\mu(\R^d)$ one has
$$
\left\|\Delta_ja\right\|_{L^\infty}\,\leq\,C\,\mu(2^{-j})
$$
for all $j\geq-1$, where the constant $C$ just depend on $\|a\|_{\mc{C}_\mu}$.
\end{prop}

Now we want to present a commutator lemma, which is fundamental in the proof of Proposition \ref{p:regular}, especially for property
\eqref{reg:source}.

First of all, let us recall the classical commutator estimates (see Lemma 2.97 of \cite{B-C-D}).
\begin{lemma} \label{l:commutator}
Let $\theta\in\mc{C}^1(\R^d)$ such that $\bigl(1+|\,\cdot\,|\bigr)\what{\theta}\,\in\,L^1$. There exists a constant $C$ such that,
for any Lipschitz function $\ell\in W^{1,\infty}(\R^d)$ and any $f\in L^p(\R^d)$ and for all $\lambda>0$, one has
$$
\left\|\bigl[\theta(\lambda^{-1}D),\ell\bigr]f\right\|_{L^p}\,\leq\,C\,\lambda^{-1}\,\left\|\nabla\ell\right\|_{L^\infty}\,\|f\|_{L^p}\,.
$$
\end{lemma}
Going along the lines of the proof, it is easy to see that the constant $C$ depends just on the $L^1$ norm
of the function $|x|\,k(x)$, where $k\,=\,\mc{F}_\xi^{-1}\theta$ denotes the inverse Fourier transform of $\theta$.

Let us give a slight variation of the previous lemma. For simplicity, we restrict our attention to the case of $\theta$ in the Schwartz
class $\mc{S}(\R^d)$: this will be enough for our aims.
\begin{lemma} \label{l:comm_L^p}
Let $\theta\in\mc{S}(\R^d)$ and $(p_1,p_2,q)\in[1,+\infty]^3$ such that $1/q\,=\,1+1/p_2-1/p_1$.
Then there exists a constant $C$ such that,
for any $f\in L^{p_1}(\R^d)$, any $\ell\in W^{1,\infty}(\R^d)$ and all $\lambda>0$,
$$
\left\|\bigl[\theta(\lambda^{-1}D),\ell\bigr]f\right\|_{L^{p_2}}\,\leq\,C\,
\lambda^{-1}\,\left\|\nabla\ell\right\|_{L^\infty}\,\|f\|_{L^{p_1}}\,.
$$
The constant $C$ just depends on the $L^q$ norm of the function $|x|\,k(x)$, where $k\,=\,\mc{F}_\xi^{-1}\theta$.
\end{lemma}
The proof follows the arguments used for the classical statement, with no special novelties. Hence, we omit it.

Let us consider now less regular functions $\ell$.
\begin{lemma} \label{l:comm-less}
Let $\theta\in\mc{C}^1(\R^d)$ be as in Lemma \ref{l:commutator}, and let $\mu$ be an admissible modulus of continuity. Then, there exists
a constant $C$ such that, for any function $\ell\in\mc{C}_{\mu}(\R^d)$ and any $f\in L^p(\R^d)$ and for all $\lambda>1$, one has
$$
\left\|\bigl[\theta(\lambda^{-1}D),\ell\bigr]f\right\|_{L^p}\,\leq\,C\,\mu(\lambda^{-1})\,\left|\ell\right|_{\mc{C}_\mu}\,\|f\|_{L^p}\,.
$$
The constant $C$ only depends on the $L^1$ norms of the functions $k(x)$ and $|x|\,k(x)$.
\end{lemma}

\begin{proof}
As in the proof of the classical result (see Lemma \ref{l:commutator}), we can write
$$
\bigl[\theta(\lambda^{-1}D),\ell\bigr]f\,=\,\lambda^d\int_{\R^d}k\bigl(\lambda(x-y)\bigr)\,f(y)\,\bigl(\ell(x)\,-\,\ell(y)\bigr)\,dy\,.
$$

Remark that the previous integral is actually taken over $\R^d\setminus\{x\}$, so that we can 
multiply and divide by $\mu(|x-y|)$. Making the seminorm $|\ell|_{\mc{C}_\mu}$ appear, thanks to Young inequality
we are reduced to estimate the quantity
$$
\lambda^d\,\left\|k(\lambda\,\cdot\,)\,\mu(\,|\cdot|\,)\right\|_{L^1}\,=\,
\lambda^d\,\int_{\R^d}|k|(\lambda\,z)\,\mu(|z|)\,dz\,.
$$
Let us split the previous integral according to the decomposition $\R^d\,=\,\left\{|z|\leq\lambda^{-1}\right\}\,\cup\,
\left\{|z|\geq\lambda^{-1}\right\}$.
For the former term, since $\mu$ is increasing we have
$$
\lambda^d\,\int_{|z|\leq \lambda^{-1}}|k|(\lambda\,z)\,\mu(|z|)\,dz\,\leq\,\mu(\lambda^{-1})\,\|k\|_{L^1}\,.
$$
For the latter term, instead, we make the non-decreasing function $\Gamma_\mu$ appear, and we estimate
\begin{eqnarray*}
\lambda^d\,\int_{|z|\geq \lambda^{-1}}|k|(\lambda\,z)\,\mu(|z|)\,dz & = &
\lambda^d\,\int_{|z|\geq \lambda^{-1}}|k|(\lambda\,z)\,\Gamma_\mu(|z|^{-1})\,|z|\,dz \\
& \leq & C\,\Gamma_\mu(\lambda)\,\lambda^{-1}\,\left\|\;|\cdot|\;k(\,\cdot\,)\right\|_{L^1}\;\leq\;C\,\mu(\lambda^{-1})\,.
\end{eqnarray*}

The lemma is hence proved.
\end{proof}

Obviously, an extension of the previous result, in the same spirit of Lemma \ref{l:comm_L^p}, holds true.

\subsection{On the BD entropy structure} \label{app:BD}

We give here the details of the proofs of some technical lemmas about BD entropy estimates. We start by proving Lemma \ref{l:F}.

\begin{proof}[Proof of Lemma \ref{l:F}]
First of all, by Lemma 2 of \cite{B-D-L} we can write
\begin{equation} \label{eq:lemma2}
\frac{1}{2}\,\frac{d}{dt}\int_\Omega\rho\,\left|\nabla\log\rho\right|^2\,+\,\int_\Omega\nabla\div u\cdot
\nabla\rho\,+\,\int_\Omega\rho\,Du:\nabla\log\rho\otimes\nabla\log\rho\,=\,0\,.
\end{equation}

Next, we multiply the momentum equation by $\nu\,\nabla\rho/\rho$ and we integrate over $\Omega$: we find
\begin{eqnarray}
& & \hspace{-1cm}
\nu\!\!\int_\Omega\left(\d_tu+u\cdot\nabla u\right)\cdot\nabla\rho\,+\,
\nu^2\int_\Omega Du:\left(\nabla^2\rho-\frac{1}{\rho}\nabla\rho\otimes\nabla\rho\right)\,+ \label{eq:mom_BD} \\
& & \qquad
+\,\frac{\nu}{\rm Ro}\int_\Omega e^3\times u\cdot\nabla\rho\,+\,
\frac{\nu}{\rm We}\int_\Omega\left|\nabla^2\rho\right|^2\,+\,
\frac{4\,\nu}{{\rm Fr}^2}\int_\Omega\bigl(P'(\rho)+P'_c\bigr)\left|\nabla\sqrt{\rho}\right|^2\,=\,0\,. \nonumber
\end{eqnarray}
Now we add \eqref{eq:lemma2}, multiplied by $\nu^2$, to \eqref{eq:mom_BD}: we end up with
\begin{eqnarray*}
& & \hspace{-1cm}
\frac{\nu^2}{2}\,\frac{d}{dt}\!\int_\Omega\rho\left|\nabla\log\rho\right|^2\,+\,
\frac{\nu}{\rm We}\!\int_\Omega\left|\nabla^2\rho\right|^2\,+\,
\frac{4\nu}{{\rm Fr}^2}\!\int_\Omega\bigl(P'+P'_c\bigr)\left|\nabla\sqrt{\rho}\right|^2\,+\,
\frac{\nu}{\rm Ro}\!\int_\Omega e^3\times u\cdot\nabla\rho\,= \\
& & \qquad\qquad =\,-\,\nu\int_\Omega\d_tu\cdot\rho\,-\,\nu^2\int_\Omega\nabla\div u\cdot\nabla\rho\,-\,
\nu\int_\Omega\left(u\cdot\nabla u\right)\cdot\nabla\rho\,-\,\nu^2\int_\Omega Du:\nabla^2\rho\,.
\end{eqnarray*}
From this relation, thanks to the mass equation and the identities
\begin{eqnarray*}
-\,\int_\Omega u\cdot\nabla\div(\rho u)\,-\,\int_\Omega\left(u\cdot\nabla u\right)\cdot\nabla\rho & = & 
\int_\Omega\rho\nabla u:\,^t\nabla u \\
-\,\int_\Omega\nabla\div u\cdot\nabla\rho\,-\,\int_\Omega Du:\nabla^2\rho & = & 0\,,
\end{eqnarray*}
we deduce the equality
\begin{eqnarray*}
& & \hspace{-1.5cm}\frac{d}{dt}F_\veps\,+\,
\frac{\nu}{\rm We}\int_\Omega\left|\nabla^2\rho\right|^2\,dx\,+\,
\frac{4\nu}{{\rm Fr}^2}\int_\Omega\bigl(P'(\rho)\,+\,P_c'(\rho)\bigr)\,\left|\nabla\sqrt{\rho}\right|^2\,dx\,+ \\
& & \qquad\qquad +\,\frac{\nu}{\rm Ro}\int_\Omega e^3\times u\cdot\nabla\rho\,dx\,=\,
-\,\nu\,\frac{d}{dt}\int_\Omega u\cdot\nabla\rho\,dx\,+\,\nu\int_\Omega\rho\,\nabla u\,:\,^t\nabla u\,dx\,.
\end{eqnarray*}
Notice that we can rewrite this last relation in the following way:
\begin{eqnarray*}
& & \hspace{-0.5cm}
\frac{1}{2}\,\frac{d}{dt}\int_\Omega\rho\,\left|u\,+\,\nu\,\nabla\log\rho\right|^2\,dx\,+\,
\frac{\nu}{\rm We}\int_\Omega\left|\nabla^2\rho\right|^2\,dx\,+\,\frac{\nu}{\rm Ro}\int_\Omega e^3\times u\cdot\nabla\rho\,dx\,+ \\
& & \qquad\qquad +\,\frac{4\nu}{{\rm Fr}^2}\int_\Omega\bigl(P'(\rho)\,+\,P_c'(\rho)\bigr)\,\left|\nabla\sqrt{\rho}\right|^2\,dx\,=\,
\frac{1}{2}\,\frac{d}{dt}\int_\Omega \rho\,|u|^2\,dx\,+\,\nu\int_\Omega\rho\,\nabla u\,:\,^t\nabla u\,dx\,.
\end{eqnarray*}

Then, we conclude by integrating with respect to time and using Proposition \ref{p:E}.
\end{proof}

Now, let us switch our attention to the proof of Lemma \ref{l:rho_BD}.

\begin{proof}[Proof of Lemma \ref{l:rho_BD}]
By Lemma \ref{l:density}, item $(i)$, we easily deduce
$$
\|\rho-1\|_{L^\infty_t(L^2)}\,\leq\,C\left(\|\rho-1\|_{L^\infty_t(L^\g)}\,+\,
\bigl(1-\mathds{1}_2(\g)\bigr)\,\left\|\nabla\rho\right\|_{L^\infty_t(L^2)}\right)\,,
$$
and the first estimate immediately follows from Proposition \ref{p:E} and Corollary \ref{c:E}.

Let us now focus on the second estimate. By Lemma \ref{l:density}, item $(ii)$, for any $0<\delta\leq1/2$ and
any $1\leq p<+\infty$ we can write
\begin{equation} \label{est:rho_BD}
\|\rho-1\|^p_{L^p_t(L^\infty)}\,\leq\,C_p\left(\|\rho-1\|_{L^\infty_t(L^2)}\,t\,+\,
\int^t_0\left\|\nabla\rho\right\|^{p((1/2)-\delta)}_{L^2}\,\left\|\nabla^2\rho\right\|^{p((1/2)+\delta)}_{L^2}\,d\tau\right)\,.
\end{equation}
Now, since $p\leq4/(1+2\delta)$, we can apply H\"older inequality to make the $L^2_t\bigl(L^2\bigr)$ norm of $\nabla^2\rho$ appear.
More precisely, using again the bounds of Proposition \ref{p:E} and Corollary \ref{c:E}, we have
\begin{eqnarray*}
\int^t_0\left\|\nabla\rho\right\|^{p((1/2)-\delta)}_{L^2}\,\left\|\nabla^2\rho\right\|^{p((1/2)+\delta)}_{L^2}\,d\tau & \leq &
\left\|\nabla\rho\right\|^{p((1/2)-\delta)}_{L^\infty_t(L^2)}\,
\int^t_0\left\|\nabla^2\rho\right\|^{p((1/2)+\delta)}_{L^2}\,d\tau \\
 & \leq & 
({\rm We})^{p((1/2)-\delta)}\;t^{1-1/q}\,\left(\left\|\nabla^2\rho\right\|^2_{L^2_t(L^2)}\right)^{1/q}\,,
\end{eqnarray*}
where we have set $q\,:=\,4/\bigl((1+2\delta)p\bigr)$.
Putting this inequality into \eqref{est:rho_BD} and using also the bound of the first part, we finally get the desired estimate.
\end{proof}

Let us conclude this section by showing the proof of Lemma \ref{l:rot}.
Notice that we are not able to exploit the presence of the cold pressure term at this level.

\begin{proof}[Proof of Lemma \ref{l:rot}]
The first inequality is trivial: we have just to write $e^3\times u\cdot\nabla\rho\,=\,
e^3\times\left(\sqrt{\rho}u\right)\cdot\nabla\sqrt{\rho}$, and then apply H\"older inequality to the integral over $\Omega$ and
Proposition \ref{p:E}.

Let us then focus on the estimate in $(ii)$. First of all, we write
\begin{equation} \label{eq:rot_parts}
\int^t_0\!\!\int_\Omega\mf{c}\,e^3\times u\cdot\nabla\rho\,=\,
\int^t_0\!\!\int_\Omega\mf{c}\,e^3\times\sqrt{\rho}u\cdot\nabla\rho\,+\,
\int^t_0\!\!\int_\Omega\mf{c}\,e^3\times\sqrt{\rho}u\cdot\nabla\rho\left(\frac{1}{\sqrt{\rho}}-1\right)\,.
\end{equation}
Now we perform an integration by parts in the latter term: denoting by $\omega=\nabla\times u$ the vorticity of the fluid,
we get
\begin{eqnarray*}
& & \hspace{-1cm} \int^t_0\!\!\int_\Omega\mf{c}\,e^3\times\sqrt{\rho}u\cdot\nabla\rho\left(\frac{1}{\sqrt{\rho}}-1\right)\;=\;
\int^t_0\!\!\int_\Omega\mf{c}\,\rho\omega^3\left(\sqrt{\rho}-1\right)\,+ \\
& & +\,\int^t_0\!\!\int_\Omega\rho u\cdot\nabla^\perp_h\mf{c}\left(\sqrt{\rho}-1\right)\,-\,
\int^t_0\!\!\int_\Omega\frac{\mf{c}}{2}\,e^3\times\sqrt{\rho}u\cdot\nabla\rho\left(\frac{1}{\sqrt{\rho}}-1\right)\,+\,
\int^t_0\!\!\int_\Omega\frac{\mf{c}}{2}\,e^3\times u\cdot\nabla\rho\,,
\end{eqnarray*}
which in turn implies
\begin{eqnarray*}
\int^t_0\!\!\int_\Omega\mf{c}\,e^3\times\sqrt{\rho}u\cdot\nabla\rho\left(\frac{1}{\sqrt{\rho}}-1\right) & = & 
\frac{2}{3}\int^t_0\!\!\int_\Omega\mf{c}\,\rho\omega^3\,\left(\sqrt{\rho}-1\right)\,+ \\
& & +\,\frac{1}{3}\int^t_0\!\!\int_\Omega\mf{c}\,e^3\times u\cdot\nabla\rho\,+\,
\frac{2}{3}\int^t_0\!\!\int_\Omega\rho u\cdot\nabla^\perp_h\mf{c}\left(\sqrt{\rho}-1\right)\,.
\end{eqnarray*}
Combining now \eqref{eq:rot_parts} with this last relation gives us
\begin{eqnarray}
\frac{\nu}{\rm Ro}\!\!\int^t_0\!\!\int_\Omega\mf{c}\,e^3\times u\cdot\nabla\rho & = & 
\frac{\nu}{\rm Ro}\int^t_0\!\!\int_\Omega\mf{c}\,\rho\omega^3\left(\sqrt{\rho}-1\right)\,+ \label{eq:rot_control} \\
& & +\,\frac{3\,\nu}{2\,{\rm Ro}}\int^t_0\!\!\int_\Omega\mf{c}\,e^3\times\sqrt{\rho}u\cdot\nabla\rho\,+\,
\frac{\nu}{\rm Ro}\int^t_0\!\!\int_\Omega\rho u\cdot\nabla^\perp_h\mf{c}\left(\sqrt{\rho}-1\right)\,. \nonumber
\end{eqnarray}

Let us start by considering the second term on the right-hand side: on the one hand, by use of H\"older inequality and
Proposition \ref{p:E}, one gets
\begin{equation}
\frac{\nu}{\rm Ro}\left|\int^t_0\!\!\int_\Omega\mf{c}\,e^3\times\sqrt{\rho}u\cdot\nabla\rho\right|\,\leq\,
\frac{C\,\nu}{\rm Ro}\int^t_0\left\|\sqrt{\rho}\,u\right\|_{L^2}\,\left\|\nabla\rho\right\|_{L^2}\,d\tau\,\leq\,
C\,\nu\,t\,\frac{\sqrt{{\rm We}}}{{\rm Ro}}\,. \label{est:rot2_g}
\end{equation}
On the other hand, if $\g=2$, by use of Young inequality we can also write
\begin{eqnarray}
\frac{\nu}{\rm Ro}\left|\int^t_0\!\!\int_\Omega\mf{c}\,e^3\times\sqrt{\rho}u\cdot\nabla\rho\right| & \leq &
\frac{C\,\nu}{\rm Ro}\int^t_0\left\|\nabla\rho\right\|_{L^2}\,d\tau\,\leq\,\frac{C\,\nu\,\sqrt{t}}{\rm Ro}
\left(\int^t_0\left\|\nabla\rho\right\|^2_{L^2}\,d\tau\right)^{\!\!1/2} \label{est:rot2_2} \\
& \leq & C\,\nu\,t\left(\frac{{\rm Fr}}{{\rm Ro}}\right)^{\!2}\,+\,
\frac{1}{2}\,\frac{\nu}{{\rm Fr}^2}\,\left\|\nabla\rho\right\|^2_{L^2_t(L^2)}\,. \nonumber
\end{eqnarray}

For the vorticity term, we start by observing that $|\sqrt{\rho}-1|\leq|\rho-1|$: therefore,
thanks also to Corollary \ref{c:E} and Lemma \ref{l:rho_BD}, we can write the estimate
\begin{eqnarray*}
\frac{\nu}{\rm Ro}\left|\int^t_0\!\!\int_\Omega\mf{c}\,\rho\omega^3\left(\sqrt{\rho}-1\right)\right| & \leq & 
\frac{\nu}{\rm Ro}\left\|\sqrt{\rho}\,Du\right\|_{L^2_t(L^2)}\,\left\|\sqrt{\rho}-1\right\|_{L^\infty_t(L^2)}\,
\left\|\sqrt{\rho}\right\|_{L^2_t(L^\infty)} \\ 
& \leq & \frac{C\,\nu}{\rm Ro}\,\zeta\,
\left\|\sqrt{\rho}\right\|_{L^2_t(L^\infty)}\,,
\end{eqnarray*}
where, for notation convenience, we have set
$$
\zeta\,=\,\zeta({\rm Fr},{\rm We})\,:=\,\left({\rm Fr}\,+\,\bigl(1-\mathds{1}_2(\g)\bigr)\,\sqrt{{\rm We}}\right)\,.
$$
In order to control the last factor in the right-hand side, we take advantage as usual of the decomposition
$\sqrt{\rho}=1+(\sqrt{\rho}-1)$;  then, applying Lemma \ref{l:rho_BD} with e.g. $\delta=1/4$ and $p=2$ (and so
$q=4/3$) implies
\begin{eqnarray*}
\left\|\sqrt{\rho}\right\|_{L^2_t(L^\infty)} & \leq & C\biggl(\bigl(1+\zeta\bigr)t\,+\,
({\rm We})^{3/4}\;\nu^{-3/4}\;t^{1/4}\left(\frac{\nu}{\rm We}\,\left\|\nabla^2\rho\right\|^2_{L^2_t(L^2)}\right)^{3/4}\biggr)^{1/2} \\
& \leq & C\biggl(\bigl(1+\zeta\bigr)^{\!1/2}\,\sqrt{t}\,+\,
({\rm We})^{3/8}\;\nu^{-3/8}\;t^{1/8}\left(\frac{\nu}{\rm We}\,\left\|\nabla^2\rho\right\|^2_{L^2_t(L^2)}\right)^{3/8}\biggr)\,.
\end{eqnarray*}
Inserting this inequality in the estimate for the vorticity term gives
\begin{eqnarray*}
\frac{\nu}{\rm Ro}\left|\int^t_0\!\!\int_\Omega\mf{c}\,\rho\omega^3\left(\sqrt{\rho}-1\right)\,dx\,d\tau\right| & \leq & 
\frac{C\,\nu}{\rm Ro}\,\zeta\,\bigl(1+\zeta\bigr)^{1/2}\,\sqrt{t}\,+ \\
& & \qquad +\,\frac{C\,\nu^{5/8}\,({\rm We})^{3/8}\,t^{1/8}}{\rm Ro}\,\zeta\,
\left(\frac{\nu}{\rm We}\,\left\|\nabla^2\rho\right\|^2_{L^2_t(L^2)}\right)^{3/8}\,,
\end{eqnarray*}
and, by application of Young inequality, in the end we find
\begin{eqnarray}
\frac{\nu}{\rm Ro}\left|\int^t_0\!\!\int_\Omega\mf{c}\,\rho\omega^3\left(\sqrt{\rho}-1\right)\,dx\,d\tau\right| & \leq & 
\frac{C\,\nu\,\sqrt{t}}{\rm Ro}\,\zeta\,\bigl(1+\zeta\bigr)^{1/2}\,+ \label{est:rot1_g2} \\
& & \quad +\,\frac{C\,\nu\,({\rm We})^{3/5}\,\zeta^{8/5}}{({\rm Ro})^{8/5}}\,t^{1/5}\,+\,
\frac{3}{8}\,\frac{\nu}{{\rm We}}\,\left\|\nabla^2\rho\right\|^2_{L^2_t(L^2)}\,. \nonumber
\end{eqnarray}

Finally, for the term involving $\nabla_h\mf{c}$ in \eqref{eq:rot_control} we can argue in a very similar way. Thanks to the bounds of
Corollary \ref{c:E} and Lemma \ref{l:rho_BD}, where this time we take $\delta=1/4$ and $p=1$ (and then $q=8/3$), we infer
\begin{eqnarray*}
& & \hspace{-1cm}
\frac{\nu}{\rm Ro}\left|\int^t_0\!\!\int_\Omega\rho u\cdot\nabla^\perp_h\mf{c}\left(\sqrt{\rho}-1\right)\right|\;\leq\;
\frac{C\,\nu}{\rm Ro}\,\left\|\sqrt{\rho}\,u\right\|_{L^\infty_t(L^2)}\,\left\|\sqrt{\rho}-1\right\|_{L^\infty_t(L^2)}\,
\left\|\sqrt{\rho}\right\|_{L^1_t(L^\infty)} \\
& & \qquad\qquad\qquad\qquad
\;\leq\;\frac{C\,\nu}{\rm Ro}\,\zeta\,\biggl(\bigl(1+\zeta\bigr)\,t\,+\,({\rm We})^{3/8}\,\nu^{-3/8}\,t^{5/8}\,
\left(\frac{\nu}{\rm We}\,\left\|\nabla^2\rho\right\|^2_{L^2_t(L^2)}\right)^{3/8}\biggr)\,.
\end{eqnarray*}
Hence, by use of Young inequality as before, it follows the control
\begin{eqnarray}
\frac{\nu}{\rm Ro}\left|\int^t_0\!\!\int_\Omega\rho u\cdot\nabla^\perp_h\mf{c}\left(\sqrt{\rho}-1\right)\right| & \leq & 
\frac{C\,\nu\,t}{\rm Ro}\,\zeta\,\bigl(1+\zeta\bigr)\,+ \label{est:rot3_g2} \\
& & \qquad +\,\frac{C\,\nu\,({\rm We})^{3/5}\,\zeta^{8/5}}{({\rm Ro})^{8/5}}\,t\,+\,
\frac{3}{8}\,\frac{\nu}{{\rm We}}\,\left\|\nabla^2\rho\right\|^2_{L^2_t(L^2)}\,. \nonumber
\end{eqnarray}

Now, we recall equality \eqref{eq:rot_control}: keeping in mind the definition of $\zeta$,
combining \eqref{est:rot1_g2} and \eqref{est:rot3_g2} with \eqref{est:rot2_g}
gives us the bound for the general case $1<\g\leq2$; the inequality in the special case $\g=2$ follows using \eqref{est:rot2_2}
instead of \eqref{est:rot2_g}.
\end{proof}

\begin{rem} \label{r:rot_parts}
The rotation term in \cite{F} was dealt with in a slightly different way, exploiting the special law of the classical component of
the pressure.
In the end, one can obtain an analogous inequality to the one given here.
\end{rem}


\end{document}